\documentclass[12pt]{article}
\usepackage[usenames]{color}
\usepackage[colorlinks=false,urlcolor=blue,citecolor=blue,linkcolor=blue,
bookmarks=true,
     bookmarksopen=false,pdftitle=created by dvipdf,pdfcreator=NM,
     pdfauthor=Megiddo,pdfsubject=mor.sty:6.29.04]{hyperref}
\usepackage{url}
\usepackage{subfigure}
\usepackage{amsfonts,mathrsfs}
\usepackage{amssymb,amsmath}
\usepackage{amsthm}
\usepackage{verbatim}
\usepackage{acronym}
\usepackage{graphicx}
\usepackage{enumerate}

\def\fskip#1{}

\newtheorem{theorem}{Theorem}

\newtheorem{corollary}{Corollary}

\newtheorem{definition}{Definition}
\newtheorem{example}{Example}

\newtheorem{lemma}{Lemma}

\newcommand{\Rm}{\mathbb{R}^m}
\newcommand{\diag}{\mbox{diag}}
\newcommand{\M}{\mathscr{M}_2}
\newcommand{\Einf}{\mathcal{E}^{\infty}}
\newcommand{\tw}{\tilde{W}}

\def\1{{\bf 1}}
\def\E{\mathcal{E}}

\def\e{{\bf e}}

\def\F{\mathscr{F}}

\newcommand{\lra}{\leftrightarrow}
\newcommand{\lrA}{\Leftrightarrow}

\def\e{\epsilon}
\def\g{\gamma}
\newcommand{\al}{\alpha}

\def\R{\mathbb{R}}
\def\re{\mathbb{R}}

\newcommand{\bS}{\mathbb{S}}

\newcommand{\EXP}[1]{\mathsf{E}\!\left[#1\right] }
\newcommand{\prob}[1]{\mathsf{Pr}\left( #1 \right)}
\newcommand{\Ginf}{G^{\infty}}
\newcommand{\W}{\{W(k)\}}

\newcommand{\C}{\mathcal{C}}
\newcommand{\Fp}{\{F(k)\}}

\begin{document}
\title{On Approximations and Ergodicity Classes in Random Chains}
\author{Behrouz Touri and Angelia Nedi\'c\thanks{Department of Industrial 
and Enterprise Systems Engineering, University of
Illinois, Urbana, IL 61801, Email: \{touri1,angelia\}@illinois.edu. 
This research is supported by the National Science Foundation 
grant CMMI 07-42538.}
}
\maketitle
\begin{abstract}
We study the limiting behavior of a random dynamic system driven by 
a stochastic chain. Our main interest is in the chains that are not necessarily
ergodic but rather decomposable into ergodic classes. To investigate 
the conditions under which the ergodic classes of a model can be identified, 
we introduce and study an $\ell_1$-approximation and infinite flow graph of 
the model. We show that the $\ell_1$-approximations of random chains preserve 
certain limiting behavior. Using the $\ell_1$-approximations, we show how 
the connectivity of the infinite flow graph is related to the structure of 
the ergodic groups of the model. Our main result of this paper provides 
conditions under which the ergodicity groups of the model can be identified
by considering the connected components in the infinite flow graph. We provide 
two applications of our main result to random networks, namely broadcast 
over time-varying networks and networks with random link failure. 
\end{abstract}

\textit{keywords:} Ergodicity, ergodicity classes, infinite flow, 
product of random matrices.

\section{Introduction}
The dynamic systems driven by stochastic matrices 
have found their use in many problems arising 
in decentralized communication~\cite{Boyd06,CarliFagnani06,CarliFagnani10,
Aysal09}, decentralized control~\cite{Jadbabaie03,Olfati-Saber04,RenBeard05,
LuChenDiscrete,LuChenContinuous},
distributed optimization~\cite{Tsitsiklis84,Tsitsiklis86,Sundhar08c,Nedic09,
Johansson08a}, and information diffusion
in social networks~\cite{Golub2010,Lobel2010}. 
In many of these applications, the ergodicity plays a central role
in ensuring that the local ``agent'' information diffuses eventually 
over the entire network of agents. The conditions under which the 
ergodicity happens have been subject of some recent 
studies~\cite{Tahbaz-Salehi08,Tahbaz-Salehi09}. 
However, the limiting behavior of the dynamics driven by time-varying chains 
has not been studied much for the case when the ergodicity is absent. 
In this context, a notable work is~\cite{Lorenz} where the limiting behavior of
the Hegselmann-Krause model~\cite{Krause}  for opinion dynamics 
has been studied. The stability of this (deterministic) model has been 
established assuming certain conditions on the model. 
Studying the stability of this model is  important 
for understanding group formation in both deterministic and random time-varying
networks, such as multiple-leaders/multiple-followers networked systems.

The main objective of this paper is to investigate the limiting behavior of the
linear dynamics driven by random independent chains of 
stochastic matrices in the absence of ergodicity. Our goal is to study 
the conditions under which the ergodic groups are formed and to characterize 
these groups. To do so, we introduce an $\ell_1$-approximation and the infinite
flow graph of a random model, and we study the properties of these objects.
Using the established properties, we extend the main result of the previous 
work in~\cite{ErgodicityPaper} to a broader class of independent random models.
We then proceed to show that for certain random models, although 
the ergodicity might not happen, the dynamics of the model still 
converges and partial ergodicity happens almost surely. In other 
words, under certain conditions, \textit{ergodic groups} are formed and 
we characterize these groups through the connected components of the infinite
flow graph. We then apply the results to a broadcast-gossip 
algorithm over a time-varying network and to a time-varying 
network with random link failures. 

The work in this paper is related to the literature on ergodicity of
random models. A discussion on the ergodicity of deterministic 
(forward and backward) chains can be  found in~\cite{SenetaMatrices}.
The earliest occurrence of the study of random models dates back to the work 
of Rosenblatt~\cite{Rosenblatt}, where the algebraic and 
topological structure of the set of stochastic matrices is employed
to investigate the limiting behavior of the product of 
independent identically distributed (i.i.d.) random matrices.
Later, in~\cite{Nawrotzki1,Nawrotzki2,Cogburn}, such a product is studied 
extensively under a more general assumption of stationarity, and 
a necessary and sufficient condition for the ergodicity is developed.
In \cite{CarliZampieri07,Tahbaz-Salehi08}  the class of i.i.d.\ random models
with almost sure positive diagonal entries were studied. In particular, 
in~\cite{Tahbaz-Salehi08} it has been showed that such a random model 
is ergodic if and only if its expected chain is ergodic. Later, this result 
has been extended to the stationary ergodic models in~\cite{Tahbaz-Salehi09}. 

Unlike the work on the i.i.d.\ models or stationary processes
in~\cite{Rosenblatt,Nawrotzki1,Nawrotzki2,Zampieri,Tahbaz-Salehi08,
Tahbaz-Salehi09}, the work in this paper is on independent 
random models that are not necessarily independent. This work is a continuation
of our work in~\cite{ErgodicityPaper}, where 
for a class of independent random models, we showed that the ergodicity is 
equivalent to the connectivity of the \textit{infinite flow graph} of 
the random model or its expected model.
Furthermore, unlike the studies that provide conditions for ergodicity 
of deterministic or random chains, such 
as~\cite{SenetaCons,Tsitsiklis84,Jadbabaie03,Cao05,Zampieri,Tahbaz-Salehi08,
Tahbaz-Salehi09,ErgodicityPaper}, the work presented in this paper 
considers the limiting behavior of deterministic and random models that are 
not necessarily ergodic. 

The main contribution of this work is in the following aspects:~(1) 
The establishment of {\it conditions on random models under which 
the ergodicity classes are fully characterized}.
This result not only implies the stability of certain random dynamics,
but also provides the structure of their equilibrium points. 
The structure is revealed through the connectivity topology 
of the infinite flow graph of the model. Although the model is not ergodic, 
the ergodicity happens {\it locally for groups of indices}, which are 
characterized as the vertices in the same connected component of the infinite 
flow graph (Theorem~\ref{thrm:extendedinfflow}).~(2) The introduction and study
of some perturbations ($\ell_1$-approximations) of chains that preserve 
ergodicity classes (as seen in Theorem~\ref{lemma:approximation}).~(3) 
The introduction of random models (class $\M$) for which 
the ergodicity can be fully characterized by the infinite flow property
(Theorem~\ref{thm:infflowthmM2}). This class encircles many of the known 
ergodic deterministic and random models, as discussed in 
Section~\ref{sec:classM}.

The structure of this paper is as follows: in Section~\ref{sec:formulation}, 
we present the problem of our interest and motivate our work by considering 
the limiting behavior of gossip protocol extended to time-varying graphs.
We also define some ergodicity notions for later use. 
In Section~\ref{sec:approximation}, we define and 
investigate $\ell_1$-approximations of random models which play 
an important role later on. 
In Section~\ref{sec:classM}, we introduce and study the ergodicity
a class of random models, and we extend the infinite flow 
theorem of~\cite{ErgodicityPaper} to this larger class of models. 
In Section~\ref{sec:infiniteflow}, we study the stability of certain
random models and characterize their ergodicity classes.
These classes are identified using infinite flow graph concept
and analysis that combines the results of 
Sections~\ref{sec:approximation} and~\ref{sec:classM},
In Section~\ref{sec:applications}, we apply
our main result to two different random models, and
we conclude in Section~\ref{sec:conclusion}.

\noindent {\bf Notation and Basic Terminology.} 
We view all vectors as columns. For a vector $x$, we write $x_i$
to denote its $i$th entry, and we write $x\ge0$ ($x>0$) to denote that 
all its entries are nonnegative (positive). 
We use $x^T$ for the transpose of a vector $x$. For a vector $x\in\Rm$, 
we use $\|x\|_p=(\sum_{i=1}^{m}|x_i|^p)^{1/p}$ for $p\geq 1$ and 
$\|x\|$ when $p=2$. For a matrix $A$, we write $\|A\|_p$ to denote the matrix 
norm induced by $\| \|_p$ vector norm. 
We use $e_i$ to denote the vector with the $i$th entry equal to~1 and 
all other entries equal to~$0$. We write 
$e$ to denote the vector with all entries equal to~$1$. 
We write $\{x(k)\}$ to denote a sequence $x(0),x(1),\ldots$ of some elements,
and we write $\{x(k)\}_{k\ge t}$ to denote the truncated sequence 
$x(t),x(t+1),\ldots$ for $t>0$.
For a given set $C$ and a subset $S$ of $C$, we write $S\subset C$ 
when $S$ is a proper subset of~$C$. A set $S\subset C$ with 
$S\ne\emptyset$ is a \textit{nontrivial} subset of $C$. 
We use $[m]$ for the integer set $\{1,\ldots,m\}$. We let
$\bar S$ denote the complement of a given set $S\subseteq[m]$
with respect to $[m]$.

We denote the identity matrix by $I$. For a finite 
collection $A_1,\ldots,A_{\tau}$ of square matrices, we write 
$A=\diag(A_1,\ldots,A_{\tau})$ to denote the block diagonal matrix with 
$r$th diagonal block being $A_r$ for $1\le r\le\tau$. For a matrix $W$, 
we write  $W_{ij}$ to denote its $(i,j)$th entry, 
$W^i$ to denote its $i$th column vector, and $W^T$ to denote its transpose.
For an $m\times m$ matrix $W$, we let 
$\sum_{i<j}W_{ij}=\sum_{i=1}^{m-1}\sum_{j={i+1}}^{m}W_{ij}$. For such a matrix 
and a nontrivial subset $S\subset[m]$, we define 
$W_S=\sum_{i\in S,j\in \bar{S}}(W_{ij}+W_{ji})$. 
A vector $v\in \Rm$ is stochastic if $v\geq 0$ and $\sum_{i=1}^{m}v_i=1$. 
A matrix $W$ is {\it stochastic} when all its rows are 
stochastic, and it is {\it doubly stochastic} when both $W$ and $W^T$ are 
stochastic. We let $\bS^m$ denote the set of $m\times m$ stochastic matrices.
We refer to a sequence $\W$ of matrices  as {\it model} or {\it chain} 
interchangeably.
 
We write $\EXP{X}$ to denote the expected value of a random variable $X$.
For an event $\mathscr{A}$, we use $\prob{\mathscr{A}}$ to denote 
its probability. For a given probability space, we say that a 
\textit{statement $R$} holds almost surely if the set of realizations
for which \textit{the statement $R$} holds is an event and 
$\prob{\{\omega\mid \mbox{statement $R$ holds}\}}=1$. 
We often abbreviate ``almost surely'' by {\it a.s.} 

\section{Problem Formulation and Motivation}\label{sec:formulation}
In this section, we describe the problems of interest and introduce 
some background concepts. We also provide an example that motivates 
the further development.

\subsection{Problem of Interest and Basic Concepts}\label{sec:problem}
We consider a linear dynamic system given by 
\begin{align}\label{eqn:dynsys}
x(k+1)=W(k)x(k)\qquad\mbox{for $k\geq t_0$},
\end{align}
where $\{W(k)\}$ is a random stochastic chain, $t_0$ is an initial time and 
$x(t_0)\in\re^m$ is an initial state of the system. 
It is well known that, for an ergodic random chain $\W$, the dynamics 
in~\eqref{eqn:dynsys} is convergent almost surely for any initial time $t_0$ 
and any initial state $x(t_0)$ (see~\cite{SenetaCons}). Furthermore, 
the limiting value of each coordinate $x_i(k)$ is the same, which is often 
referred to as consensus, agreement, or synchronization.  In this case, 
the sequence $\W$ has a single ergodic class that consists of all coordinate 
indices $\{1,\ldots,m\}$. In other words,  the dynamics in~\eqref{eqn:dynsys} 
is stable and the equilibrium points lie on the line spanned by the vector $e$.
 
A natural question arises: what happens if $\W$ is not ergodic?
In particular, what can we say
about the limiting dynamics of the coordinates $x_i(t)$? 
What can we say about the stability and the equilibrium points of the dynamic 
system~\eqref{eqn:dynsys}? Can we determine the ergodicity classes based on 
the properties of the matrices $W(k)$? Our motivation in this paper is to 
answer these questions. To do this, we first investigate these questions for a 
more structured dynamic, namely the \textit{gossip algorithm on a time-varying 
network}. Then, we show that the results are in fact 
applicable  to a more general class of random dynamics. 

We next formalize several notions related to random chains. 
Let $(\Omega,\mathcal{F},\prob{\cdot})$ be a probability space and 
let $\mathbb{Z}^+$ be the set of non-negative 
integers. Let $W:\Omega\times \mathbb{Z}^+\rightarrow \bS^m$ be a 
random matrix process such that $W_{ij}(k)$ is a Borel-measurable function for 
all $i,j\in[m]$ and $k\geq 0$. We refer to such a process as 
a \textit{random chain} or a \textit{random model}. We often denote such 
a process by its coordinate representation $\W$. If matrices $W(k)$ are 
independent, the model is independent. In addition, if $W(k)$s 
are identically distributed, $\W$ is an 
\textit{independent identically distributed (i.i.d.)} random model.

We now introduce the concept of ergodicity. We first define it for 
a deterministic chain $\{A(k)\}$, which can be viewed 
as a special independent random chain by setting 
$\Omega=\{\omega\}$, $\mathcal{F}=\{\{\omega\},
\emptyset\}$, $\prob{\{\omega\}}=1$ and $W(k)(\omega)=A(k)$. Then, 
the dynamic system in Eq.~\eqref{eqn:dynsys} is deterministic and
we have the following definition.
\begin{definition}\label{def:ergodicity}
A chain $\{A(k)\}$ is an \textit{ergodic chain} or 
an \textit{ergodic model} if $\lim_{k\rightarrow\infty}(x_i(k)-x_j(k))=0$ 
for any $i,j\in[m]$, any starting time $t_0\geq 0$ and any
starting point $x(t_0)\in\Rm$. The chain 
\textit{admits consensus} if the above assertion is true for $t_0=0$. 
\end{definition}

Note that, for a random model $\W$, 
we can speak about subsets  $\E$ and $\C$ of $\Omega$ on 
which the ergodicity and consensus happen, respectively. These sets 
are given by
\[\E=\cap_{t_0=0}^{\infty}\left(\cap_{\ell=1}^{m}\{\omega\mid 
\lim_{k\rightarrow\infty}(x_i(k)-x_j(k))=0\ \ 
\mbox{for all $i,j\in[m]$, $x(t_0)=e_\ell$}\} \right),\] 
\[\C=\cap_{\ell=1}^{m}\{\omega\mid 
\lim_{k\rightarrow\infty}(x_i(k)-x_j(k))=0\ \ 
\mbox{for all $i,j\in[m],x(0)=e_\ell$}\}.\] 
The scalars $x_i(k)-x_j(k)$ are random variables since $W_{ij}(k)$ are 
Borel-measurable, so that  $\E$ and $\C$ are events (see~\cite{ErgodicityPaper}
for a discussion on why it suffices to consider only $x(0)=e_\ell$).
The random model is  {\it ergodic} ({\it admits consensus})
if the event $\E$ ($\C$) happens almost surely.

The ergodicity of certain random models is closely related to the 
\textit{infinite flow property}, as shown in~\cite{ErgodicityPaper,CDC2010}.
We will use this property, so we recall its definition below.

\begin{definition}\label{def:inflowprop}(\textit{Infinite Flow Property}) 
A deterministic chain $\{A(k)\}$ has infinite flow property 
if $\sum_{k=0}^{\infty}A_S(k) =\infty$ for any nonempty $S\subset [m]$. 
A random model $\W$ has infinite flow property if it has 
infinite flow property almost surely.
\end{definition}

As in the case of consensus and ergodicity events, 
the subset of $\Omega$ over which the infinite flow happens is an event 
since $W_{ij}(k)$s are Borel-measurable. We denote this event by $\F$. 

In our further development, we also use some additional properties of 
random models such as weak feedback property and a common steady state in 
expectation, as introduced in~\cite{ErgodicityPaper}. For convenience, 
we provide them in the following definition.

\begin{definition}\label{def:feedbacksteadystate}
Let $\W$ be a random model. We say that the model has:
\begin{enumerate}[(a)]
\item  \textit{Weak feedback property} 
if there exists $\gamma>0$ such that \\
\centerline{
$\EXP{W^i(k)^TW^j(k)}\geq \gamma(\EXP{W_{ij}(k)}+\EXP{W_{ji}(k)})
\qquad\hbox{for all $k\geq 0$ and $i\not=j$, $i,j\in[m]$}.$} 
\item 
A \textit{common steady state $\pi$ in expectation} 
if $\pi^T\EXP{W(k)}=\pi^T$ for all $k\geq 0$.
\end{enumerate}
\end{definition}
Any random model with $W_{ii}(k)\geq \gamma>0$ 
almost surely for all $k\geq 0$ and $i\in [m]$ has weak feedback property. 
Also, the i.i.d.~models with almost sure positive diagonal entries have weak 
feedback property, as seen in~\cite{ErgodicityPaper}. 
As an example of a model with a common steady state $\pi$ in expectation, 
consider any i.i.d.\ random model $\W$. Another example
is a model where each $W(k)$ is doubly stochastic almost 
surely, for which we have $\pi=\frac{1}{m}\,e$. 

With a given random model, we associate an undirected graph which 
we refer to as {\it infinite flow graph}. We define this graph as 
a simple undirected graph with links that have sufficient information flow
(simple graph is a graph without self-loops and multiple edges).  

\begin{definition}\label{def:infiniteflowgraph}(\textit{Infinite Flow Graph}) 
The infinite flow graph of a random model $\W$ is 
the graph $\Ginf=([m],\Einf)$, where $\{i,j\}\in\Einf$ 
if and only if $\sum_{k=0}^{\infty}(W_{ij}(k)+W_{ji}(k))=\infty$ 
almost surely. 
\end{definition}  

The infinite flow graph has been (silently) used  in~\cite{Tsitsiklis84}
mainly to establish the ergodicity of a certain deterministic chains. 
Here, however, we make use of this graph to establish ergodicity classes for 
certain independent random chains. In particular, as we will see in the 
later sections, the infinite flow graph and its connected components play 
important role in identifying the ergodicity classes of the model. 
In the sequel, {\it a connected component of a graph will always be maximal} 
with respect to the set inclusion, i.e., it will be 
the connected subgraph of the given graph that is not 
properly contained in any other connected subgraph.

\subsection{Infinite Flow Graph and
Gossip Algorithm on Time-Varying Network}\label{sec:gossip}
To illustrate the use of the infinite flow graph in determining the
ergodicity classes of a random model, we consider a gossip algorithm over a 
time-varying network that is discussed in~\cite{ErgodicityPaper,CDC2010} 
as an extension of the original gossip algorithm \cite{Boyd05,Boyd06}. 
We investigate the algorithm for the case when the underlying network is not
connected. We assume that there are $m$ agents
that communicate over an undirected graph $G=([m],E)$ with link set $E$.
The agent communication at any time $k$ is determined by a matrix
$P(k)$ such that $\sum_{i<j}P_{ij}(k)=1$. 
The value $P_{ij}(k)$ is the probability that 
the link $\{i,j\}\in E$ is activated, independently of the link realizations in
the past, and $P_{ij}(k)=0$ if $\{i,j\}\not\in E$.
When the link $\{i,j\}$ is activated at time $k$, agents $i$ and $j$ 
exchange their values and update as follows:
\begin{align*}
x_\ell(k+1) & =\frac{1}{2}\,\left(x_i(k)+ x_j(k)\right)
\qquad\hbox{for }\ell=i,j,
\end{align*} 
while the other agents do not update, i.e.,
$x_\ell(k+1)=x_\ell(k)$ for $\ell\ne i,j$. Here, $x_i(0)$ is the initial value 
of agent $i$. Thus, we have a dynamic model of the form~\eqref{eqn:dynsys}, 
where
\begin{align}\label{eqn:extendedgossip}
W(k)=I-\frac{1}{2}(e_i-e_j)(e_i-e_j)^T\qquad
\mbox{with probability $P_{ij}(k)$.}
\end{align}
By looking at the connected components of the infinite flow graph of
the model $\W$, we can characterize the limiting behavior
of the dynamics $\{x(k)\}$. Specifically, let 
$\Ginf$ be the infinite flow graph of the model.
Let $S_1,\ldots,S_{\tau}\subset [m]$ be the connected components
of $\Ginf$, where $\tau\ge1$ is the number of 
the components. We have the following result.

\begin{theorem}\label{thm:gossip}
Consider the time-varying gossip algorithm given by~\eqref{eqn:extendedgossip}.
Then, for any initial vector $x(0)\in\Rm$, the dynamics $\{x(k)\}$ 
converges almost surely. Furthermore, we have 
$\lim_{k\rightarrow\infty}(x_i(k)-x_j(k))=0$ almost surely for $i,j\in S_r$ 
and $r\in [\tau]$.
\end{theorem}

The proof of Theorem~\ref{thm:gossip} relies on the infinite flow theorem, 
which was established in~\cite{ErgodicityPaper}
and it is provided below for easier reference. We use the theorem to
derive a special consequential result and, then, we proceed with the proof
of Theorem~\ref{thm:gossip}.

\begin{theorem}\label{thm:infflowthm} ({\it Infinite Flow Theorem}) \ 
Let a random model $\{W(k)\}$ be independent, and 
have a common steady state $\pi>0$ in expectation and
weak feedback property. 
Then, the following conditions are equivalent: 
\begin{enumerate}
\item[(a)] The model is ergodic.
\item[(b)] The model has infinite flow property.
\item[(c)] The expected model has infinite flow property. 
\item[(d)] The expected model is ergodic.
\end{enumerate}
\end{theorem}
 
When the conditions of Theorem~\ref{thm:infflowthm} are satisfied 
in almost sure sense, i.e., they hold for almost all sample paths of the random
dynamics, we obtain the following result.

\begin{corollary}\label{cor:pathproperty}
Let $\W$ be a random model. Suppose that $\pi^T W(k)=\pi^T$ almost surely for 
a vector $\pi>0$ and all $k\geq 0$. Also, suppose that 
$W^i(k)^TW^j(k)\geq \gamma(W_{ij}(k)+W_{ji}(k))$ almost surely for 
some $\g>0$ and for all $k\geq 0$ and $i,j\in[m]$ with $i\not=j$. 
Then, the infinite flow and the ergodicity events coincide almost surely, 
i.e., $\E=\F$ almost surely.
\end{corollary}

We are now ready to present the proof of Theorem~\ref{thm:gossip}.

\begin{proof}
Let $m_r=|S_r|$ be the number of vertices in the connected component $S_r$ of 
$\Ginf$. Without loss of generality assume that the vertices are ordered 
such that $S_1=\{1,\ldots,a_1\},$ $S_2=\{a_1+1,\ldots,a_2\},\ldots,S_{\tau}
=\{a_{\tau-1}+1,\ldots,a_{\tau}\}$ where $m_r=a_r-a_{r-1}$ and 
$a_0=0<a_1<a_2<\cdots<a_{\tau-1}<a_{\tau}=m$. By the Borel-Cantelli 
lemma~\cite{Durrett}, page 46, for almost 
all $\omega\in \Omega$ there exists $N(\omega)<\infty$ such that 
no communication link $\{i,j\}$ will appear between two different components 
$S_{\al}$ and $S_{\beta}$ for $\al\not=\beta$ and $\al,\beta\in [\tau]$ for any
time $k\geq N(\omega)$. Therefore, for almost all $\omega\in \Omega$, the chain
$\{W(k)\}(\omega)$ can be written in the form $W(k)(\omega)
=\diag(W^{(1)}(k)(\omega),\ldots,W^{(\tau)}(k)(\omega))$ for 
$k\geq N(\omega)$ where $W^{(r)}(k)(\omega)$ are $m_r\times m_r$ stochastic 
matrices. From the dynamic system perspective, this means that 
after time $N(\omega)$, the dynamics $\{x(k)\}$ driven by 
$\{W(k)\}$ can be decoupled into $\tau$ disjoint dynamics 
$\{x^{(r)}(k)\}$, where $\{x^{(r)}(k)\}\subset\mathbb{R}^{m_r}$ 
is governed by $\{W^{(r)}(k)(\omega)\}$. By the Borel-Cantelli 
lemma~\cite{Durrett}, page 46, the models 
$\{W^{(r)}(k)(\omega)\}_{k\geq N(\omega)}$ have infinite flow property 
for $r\in[\tau]$. Now, for every $\omega$ and 
each $r\in[\tau]$, the deterministic model 
$\{W^{(r)}(k)(\omega)\}_{k\geq N(\omega)}$ is doubly stochastic and 
$W^{(r)}_{ii}(k)(\omega)\geq \frac{1}{2}$ for all $i\in[m]$ and $k\geq 0$. 
Therefore, by Corollary~\ref{cor:pathproperty}, the individual chains are 
ergodic almost surely. Hence, starting from any $x(0)\in\Rm$, we have 
$x(N(\omega))=\Phi(N(\omega),0)x(0)$
From time $N(\omega)$, the $m_r$ coordinates of $x(k)$ 
that belong to $S_r$ evolve by the chain $\{W^{(r)}(k)\}$. Since each 
chain $\{W^{(r)}(k)\}$ is ergodic, it follows that 
$\lim_{k\rightarrow\infty}(x_i(k)-x_j(k))=0$ for $i,j\in S_r$ and $r\in[\tau]$.
Note that although the time $N(\omega)$ is a
random (stopping) time, the sets $S_r$ are deterministic. 
\end{proof}

 Theorem~\ref{thm:gossip} states that 
any two agents in the same connected component of $\Ginf$ will 
consent for any initial point $x(0)\in\Rm$. 
In the upcoming sections, we develop tools to provide results 
similar to Theorem~\ref{thm:gossip} but for a larger class of random models.
In the development, we use some refinements of the ergodicity notion, 
as discussed in the following section.

\subsection{Mutually Weakly Ergodic and Mutually Ergodic Indices}\label{sec:me}
Departing from the ergodicity and moving toward the 
ergodic groups of a model, we make use of some refinements of
the ergodicity concept, as given in the following definition.

\begin{definition}\label{def:differentergodic}
Let $\{A(k)\}$ be a deterministic chain, and let $\{x(k)\}$
in~\eqref{eqn:dynsys} be driven by $\{A(k)\}$. We say that:
\begin{enumerate}[(a)]
\item  
Two indices $i,j\in[m]$ are \textit{mutually weakly ergodic indices} for
the chain if 
$\lim_{k\rightarrow\infty}(x_i(k)-x_j(k))=0$ for any initial time $t_0\geq 0$ 
and any initial point $x(t_0)\in\Rm$. We write $i\lra_A j$ when 
$i$ and $j$ are mutually weakly ergodic indices for the chain $\{A(k)\}$.
\item 
The index $i\in[m]$ is an \textit{ergodic index} for the chain 
if $\lim_{k\rightarrow\infty}x_i(k)$ exists for any starting time $t_0\geq 0$ 
and any initial point $x(t_0)\in\Rm$. 
The chain $\{A(k)\}$ is \textit{stable} when each $i\in [m]$ is 
an ergodic index. 
\item 
Two indices $i,j\in[m]$ are \textit{mutually ergodic indices} if $i$ and $j$ 
are ergodic indices and $i\lra_A j$ for any initial time $t_0\geq 0$ and 
any initial point $x(t_0)\in\Rm$. We write $i\lrA_A j$ when 
$i$ and $j$ are mutually ergodic indices for the chain $\{A(k)\}$.
\end{enumerate}
\end{definition}

The relation $\lra_A$ is an 
equivalence relation on $[m]$ and we can consider its equivalence classes.
We refer to the equivalence classes of this relation as 
\textit{ergodicity classes} and to the resulting partitioning of $[m]$ 
as the \textit{ergodicity pattern} of the model. 

Definition~\ref{def:differentergodic} extends naturally to a random model.
Specifically, if any of the properties in Definition~\ref{def:differentergodic}
holds almost surely for a random model $\W$, we say that 
the model $\W$ has the corresponding property.
In the further development, when unambiguous, 
we will omit the explicit dependency of the relation $\lra$ 
and $\lrA$ on the underlying chain. 

For a random model $\W$, the set of realizations for which 
$i\lra j$ (or $i\lrA j$) is a measurable set; hence, an event. 
When the model $\W$ is independent, these events are tail events 
and, therefore, each of these events happens with either probability zero or
one. Hence, for an independent random model $\W$, we write $i\lra j$ 
when $\prob{i\lra j}=1$ and $i\not\lra j$ when $\prob{i\lra j}=0$.
Analogously, we define  $i\lrA j$ and $i\not\lrA j$.
Thus, the ergodicity pattern of any independent random model is well-defined. 

In the light of Definition~\ref{def:differentergodic} and the above discussion,
by recalling the definition of the ergodicity
(Definition~\ref{def:ergodicity}), we see that a chain is ergodic 
if and only if its ergodicity class is a singleton or, equivalently,
its ergodicity pattern is $\{[m]\}$. 
Furthermore, we can interpret Theorem~\ref{thm:gossip} as follows:
for the gossip algorithm of Eq.~\eqref{eqn:extendedgossip}, we have
$i\lrA_W j$ if $i$ and $j$ belong to the same connected component of 
the infinite flow graph $\Ginf$. 
We prove in Section~\ref{sec:infiniteflow} that this result 
holds for any random model satisfying the conditions of the infinite flow 
theorem (Theorem~\ref{thm:infflowthm}). We actually show a stronger result
stating that $i\lrA_W j$ if and only if $i$ and $j$ 
belong to the same connected component of the infinite flow graph $\Ginf$.
We also prove that $i\lrA_W j$ if and only if $i\lrA_{\bar W} j$,
where $i\lrA_{\bar W} j$ is the relation for the expected chain
$\{\EXP{W(k)}\}$. To show these results, we use  
the concept of $\ell_1$-approximations and 
their properties (developed in Section~\ref{sec:approximation}),
and the concept of $\M$-class of random chains 
with their basic properties (established in Section~\ref{sec:classM}). 

\section{Approximation of chains}\label{sec:approximation}
Here, we consider an approximation of chains
that preserves ergodicity classes for the indices of the chains.
This approximation plays a key role in our study of non-ergodic chains
in Section~\ref{sec:infiniteflow}. 
In what follows, we say that two chains $\{W(k)\}$ and $\{U(k)\}$ 
have the {\it same ergodicity classes} if there 
exists a bijection $\theta:[m]\rightarrow[m]$ between the indices of 
$\{W(k)\}$ and $\{U(k)\}$ such that: 
\begin{itemize}
\item[(a)] $i\lra_W j$ if and only if $\theta(i)\lra_U \theta(j)$, and
\item[(b)]
$i$ is an ergodic index for $\W$ if and only if $\theta(i)$ is an ergodic 
index for $\{U(k)\}$.
\end{itemize} 
When  $\theta$ is a  bijection, then the indices of one of the chains
can be permuted according to the bijection $\theta$, so that 
the bijection $\theta$ can always be taken as identity.
We assume that this is the case for the rest of the paper.

We want to determine a perturbation of a chain 
that does not affect the ergodicity classes of the chain. 
It turns out that a perturbation that is small in $\ell_1$-norm has such 
a property. To formally set up the framework for this development,
we next introduce the concept of $\ell_1$-approximation.  
\begin{definition}\label{def:l1aprox}
A deterministic chain $\{B(k)\}$ is an $\ell_1$-approximation of a chain 
$\{A(k)\}$ if 
$\sum_{k=0}^{\infty}|A_{ij}(k)-B_{ij}(k)|<\infty$ for all $i,j\in [m]$. 
\end{definition}

As an example of such chains,
consider two models $\{A(k)\}$ and $\{B(k)\}$ that differ only 
at finitely many instances, i.e., there exists some time $t\ge0$
such that $W(k)=U(k)$ for all $k\geq t$. 
Since every entry in each of the matrices $A(k)$ and 
$B(k)$ is in the interval $[0,1]$, it follows that 
$\sum_{k=0}^{\infty}|A_{ij}(k)-B_{ij}(k)|\leq t$. Hence, the two models are 
$\ell_1$-approximation of each other. We will use such an approximation for 
random models (in the proof of Theorem~\ref{thrm:extendedinfflow}), 
so we need to extend this notion to random models.

Definition~\ref{def:l1aprox} extends to random chains by requiring that
$\ell_1$-approximation is almost sure. Specifically, a random chain $\{U(k)\}$ 
is an $\ell_1$-approximation of a random chain $\{W(k)\}$ if 
$\sum_{k=0}^{\infty}|W_{ij}(k)-U_{ij}(k)|<\infty$ almost surely 
for all $i,j\in [m]$.

We have some remarks for Definition~\ref{def:l1aprox}. First, 
we note that $\ell_1$-approximation is an equivalence relation for 
deterministic chains, since the set of all absolutely summable sequences in 
$\R$ is a vector space over $\R$. This is also true
for independent random chains $\{W(k)\}$ and 
$\{U(k)\}$ that are adapted to the same sigma-field. In this case,
we have $\sum_{k=0}^{\infty}|W_{ij}(k)-U_{ij}(k)|<\infty$ for all $i,j\in[m]$ 
with either probability zero or one, due to Kolmogorov's 0-1 law 
(\cite{Durrett}, page~61). Thus, $\{W(k)\}$ and $\{U(k)\}$ are 
$\ell_1$-approximations of each other with either probability zero or one.
Second, we note that there are alternative formulations of 
$\ell_1$-approximation. Since the matrices have a finite dimension, if 
$\sum_{k=0}^{\infty}|A_{ij}(k)-B_{ij}(k)|<\infty$ for all $i,j\in[m]$, 
then $\sum_{k=0}^{\infty}\|A(k)-B(k)\|_p<\infty$ for any $p\geq 1$.
Thus, an equivalent definition of $\ell_1$-approximation is obtained 
by requiring that 
$\sum_{k=0}^{\infty}\|A(k)-B(k)\|_p<\infty$ for some $p\geq 1$. 

To illustrate how we can construct an $\ell_1$-approximation,
we consider the time-varying gossip model of Section~\ref{sec:formulation}. 
Let $\W$ be the chain of the model as given in Eq.~\eqref{eqn:extendedgossip},
and let $\Ginf$ be its infinite flow graph. 
Assume that the connected components of $\Ginf$ are $S_1,\ldots,S_{\tau}$. 
Now, define the \textit{approximate gossip model} $\{U(k)\}$ as follows:
\begin{align}
U(k)=\left\{
\begin{array}{ll}
W(k)&\mbox{if link $\{i,j\}$ is activated at time $k$ with 
$i,j\in S_r$ for some $r$},\\
I&\mbox{otherwise}.
\end{array}\right.
\end{align}
Basically, in the approximate gossip model, we cut the links between 
the agents that belong to different connected components of $\Ginf$. 
In this case, by the definition of the infinite flow graph
we have $\sum_{k=0}^{\infty}|W_{ij}(k)-U_{ij}(k)|<\infty$ almost surely 
for all $i$ and $j$ that do not belong to the same connected component
of $\Ginf$. In this way, we have an approximate dynamic consisting of $\tau$ 
decoupled dynamics (one per connected component of $\Ginf$). 
At the same time, the original and the approximate dynamics have the same
ergodicity classes. This will follow from the forthcoming 
Lemma~\ref{lemma:approximation}, which shows the result
for a more general model. 

Now, we present Lemma~\ref{lemma:approximation} which establishes
the main result of this section. The lemma states that 
if two chains are $\ell_1$-approximations of each other, 
then their ergodicity classes are identical. 

\begin{lemma}\label{lemma:approximation}
(\textit{Approximation Lemma})
Let a deterministic chain $\{B(k)\}$ be an $\ell_1$-approximation of 
a deterministic chain $\{A(k)\}$. Then, the models have the same ergodicity
classes.
\end{lemma}
\begin{proof}
Suppose that $i\lra_B j$. Let $t_0=0$ and let $x(0)\in[0,1]^m$. 
Also, let $\{x(k)\}$ be the dynamics as defined in 
Eq.~\eqref{eqn:dynsys} by matrices $\{A(k)\}$. For any $k\ge 0$, we have 
\[x(k+1)=A(k)x(k)=(A(k)-B(k))x(k)+B(k)x(k).\]
Since $|x_i(k)|\leq 1$ for any $k\ge0$ and any $i\in[m]$, it follows that
for all $k\ge0$,
\begin{align}\label{eqn:l1approx}
\|x(k+1)-B(k)x(k)\|_{\infty}\leq \|A(k)-B(k)\|_{\infty}.
\end{align}

We want to show that $i\lra_A j$, or equivalently that
$\lim_{k\to\infty}(x_i(k)-x_j(k))=0$. To do so, 
we let $\epsilon>0$ be arbitrary but fixed. 
Since $\{B(k)\}$ is an $\ell_1$-approximation of $\{A(k)\}$, 
there exists time $N_{\epsilon}\geq 0$ such that 
$\sum_{k=N_{\epsilon}}^{\infty}\|A(k)-B(k)\|_{\infty}\le\epsilon$. 
Let $\{z(k)\}_{k\geq N_{\epsilon}}$ be the dynamics given by 
Eq.~\eqref{eqn:dynsys}, which is driven by $\{B(k)\}$ 
and started at time $N_\epsilon$ with the initial 
vector $z(N_\epsilon)=x(N_\epsilon)$. 
We next show that 
\begin{align}\label{eqn:sum}
\|x(k+1)-z(k+1)\|_{\infty}
\leq \sum_{t=N_{\epsilon}}^{k}\|A(t)-B(t)\|_{\infty}
\qquad \hbox{for all $k\ge N_\epsilon$}.\end{align}
We use the induction on $k$, so we consider $k=N_{\epsilon}$. Then, 
by Eq.~\eqref{eqn:l1approx}, we have 
$\|x(N_{\epsilon}+1)-B(N_{\epsilon})x(N_{\epsilon})\|_{\infty}
\leq \|A(N_\epsilon)-B(N_\epsilon)\|_{\infty}$. Since 
$z(N_{\epsilon})=x(N_{\epsilon})$,  it follows that
$\|x(N_{\epsilon}+1)-z(N_{\epsilon}+1)\|_{\infty}
\leq \|A(N_\epsilon)-B(N_\epsilon)\|_{\infty}$. 
We now assume that 
$\|x(k)-z(k)\|_{\infty}\leq 
\sum_{t=N_{\epsilon}}^{k-1}\|A(t)-B(t)\|_{\infty}$ for some $k>N_\epsilon$. 
Using Eq.~(\ref{eqn:l1approx}) and the triangle inequality, we have
\begin{align}\nonumber
\|x(k+1)-z(k+1)\|_{\infty}&=\|A(k)x(k)-B(k)z(k)\|_{\infty}\cr
&=\|(A(k)-B(k))x(k)+B(k)(x(k)-z(k))\|_{\infty}\cr 
&\leq \|(A(k)-B(k))\|_{\infty}\|x(k)\|_{\infty}
+\|B(k)\|_{\infty}\|(x(k)-z(k))\|_{\infty}.  \end{align}
By the induction hypothesis and relation $\|B(k)\|_{\infty}=1$,
which holds since $B(k)$ is a stochastic matrix, it follows that
$\|x(k+1)-z(k+1)\|_{\infty}
\leq\sum_{t=N_{\epsilon}}^{k}\|A(t)-B(t)\|_{\infty}$,
thus showing relation~\eqref{eqn:sum}.

Recalling that the time $N_{\epsilon}\geq 0$ is such that 
$\sum_{k=N_{\epsilon}}^{\infty}\|A(k)-B(k)\|_{\infty}\le\epsilon$
and using relation~\eqref{eqn:sum}, we obtain
for all $k\ge N_\e$, 
\begin{align}\label{eqn:l1difference}
\|x(k+1)-z(k+1)\|_{\infty}
\leq\sum_{t=N_{\epsilon}}^{k}\|A(t)-B(t)\|_{\infty}
\leq\sum_{t=N_{\epsilon}}^{\infty}\|A(t)-B(t)\|_{\infty}\leq \epsilon.
\end{align}
Therefore, $|x_i(k)-z_i(k)|\leq \epsilon$ and 
$|z_j(k)-x_j(k)|\leq \epsilon$ for any  $k\ge N_\e$,
and by the triangle inequality we have 
$|(x_i(k)-x_j(k))+(z_i(k)-z_j(k))|\leq 2\epsilon$ for any 
$k\geq N_{\epsilon}$. Since $i\lra_B j$, it follows that 
$\lim_{k\rightarrow\infty}(z_i(k)-z_j(k))=0$ and 
$\limsup_{k\rightarrow\infty}|x_i(k)-x_j(k)|\leq 2\epsilon.$
The preceding relation holds for any $\e>0$, 
implying that $\lim_{k\rightarrow\infty}(x_i(k)-x_j(k))=0.$
Furthermore, the same analysis would go through when $t_0$ is arbitrary
and the initial point $x(0)\in\re^m$ is arbitrary with 
$\|x(0)\|_\infty\ne1.$ Thus, we have $i\lra_A j$. 

Using the same argument and inequality~\eqref{eqn:l1difference}, 
one can deduce that if $i$ is ergodic index for $\{B(k)\}$, then it is also
an ergodic index for $\{A(k)\}$. Since $\ell_1$-approximation is symmetric
with respect to the chains, the result follows. 
\end{proof}

Approximation lemma states that the ergodicity classes coincide
for two chains that are $\ell_1$-approximations of each other. 
However, the lemma does not say that the limiting values of the chains have 
to be the same for any of the ergodic classes. Within an ergodic class, 
the limiting values may differ. In particular, consider an ergodic chain
$\{A(k)\}$ and its $\ell_1$-approximation chain $\{B(k)\}$. By the 
the approximation lemma it follows that $B\{(k)\}$ is also ergodic.
Thus, we conclude that the class of ergodic deterministic chains is closed 
under $\ell_1$-approximations. 
The same is true for the class of ergodic independent random models. 

Approximation lemma is a tight result with respect to the choice of 
$\ell_1$-norm. In other words, the lemma need not hold if 
we consider $\ell_p$-approximation with $p>1$.
To see this, consider the following $2\times 2$ chain:
\begin{align}\nonumber
A(k)=\left[\begin{array}{cc}
1-\frac{1}{k+2}&\frac{1}{k+2}\\
\frac{1}{k+2}&1-\frac{1}{k+2}
\end{array}
\right]\qquad \mbox{for all $k\geq 0$}.
\end{align}
The chain is doubly stochastic and $A_{ii}(k)\geq \frac{1}{2}$ for $i=1,2$ 
and all $k\geq 0$. Thus, it has weak feedback property with 
a feedback coefficient $\gamma=\frac{1}{2}$. Also, since 
$\sum_{k=0}^{\infty}\frac{1}{k+2}=\infty$, by the infinite flow theorem 
(Theorem~\ref{thm:infflowthm}), the chain is ergodic. 
For any $p>1$, the identity chain $\{I\}$ is an $\ell_p$-approximation 
of $\{A(k)\}$, i.e., 
$\sum_{k=0}^\infty|A_{ij}(k)-I_{ij}|^p
=\sum_{k=0}^\infty \frac{1}{(k+2)^p}<\infty$ for $i,j=1,2$.
However, the chain $\{I\}$ is not ergodic and, therefore, 
$\ell_1$-norm in the approximation lemma cannot be replaced by any 
$\ell_p$-norm for $p>1$.

We now present an important result which is a more involving consequence
of Lemma~\ref{lemma:approximation}. This result relates mutual weak ergodicity 
of a model to the connected components in the infinite flow graph of the model
(see Definition~\ref{def:infiniteflowgraph}).
The result plays a crucial role in the characterization of the ergodicity
of some chains (in forthcoming Theorem~\ref{thm:sufficient}). 
\begin{lemma}\label{lemma:extendedinfinite}
Let $\{A(k)\}$ be a deterministic chain and let $\Ginf$ be its
infinite flow graph. Then, $i\lra_A j$ implies that 
$i$ and $j$ belong to the same connected component of $\Ginf$.
\end{lemma}
\begin{proof}
To arrive at a contradiction, suppose that $i$ and $j$ belong to 
two different connected components $S,T\subset [m]$ of $\Ginf$. 
Therefore, $T\subset \bar{S}$ implying that $\bar{S}$ is not empty. 
Also, since $S$ is a connected component of $\Ginf$, it follows that 
$\sum_{k=0}^{\infty}A_S(k)<\infty$. Without loss of generality, we assume 
that $S=\{1,\ldots,i^*\}$ for some $i^*<m$, and consider the chain $\{B(k)\}$
defined by
\begin{align}
B_{ij}(k)=\left\{
\begin{array}{ll}
A_{ij}(k)&\mbox{if $i\not=j$ and $i,j\in S$ or $i,j\in \bar{S}$},\\
0&\mbox{if $i\not=j$ and $i\in S,j\in\bar{S}$ or $i\in\bar{S},j\in{S}$},\\
A_{ii}(k)+\sum_{\ell\in\bar{S}}A_{i\ell}(k)&\mbox{if $i=j\in S$},\\
A_{ii}(k)+\sum_{\ell\in{S}}A_{i\ell}(k)&\mbox{if $i=j\in \bar{S}$}.
\end{array}\right.
\end{align}
The above approximation simply sets the cross terms between 
$S$ and $\bar{S}$ to zero, and adds the corresponding values to the diagonal 
entries to maintain the stochasticity of the matrix $B(k)$. 
Therefore, for the stochastic chain $\{B(k)\}$ we have
\[B(k)=\left[\begin{array}{cc}
B_1(k)&0\\0&B_2(k)
\end{array}\right],\]
where $B_1(k)$ and $B_2(k)$ are respectively $i^*\times i^*$ and 
$(m-i^*)\times(m-i^*)$ matrices for all $k\geq 0$. By the assumption 
$\sum_{k=0}^{\infty}A_S(k)<\infty$, the chain $\{B(k)\}$ is an 
$\ell_1$-approximation of $\{A(k)\}$. 
Now, let $u_{i^*}$ be the vector which has the first $i^*$ coordinates equal to
one and the rest equal to zero, i.e.,  $u_{i^*}=\sum_{\ell=1}^{i^*}e_{\ell}$. 
Then, $B(k)u_{i^*}=u_{i^*}$ for any $k\geq 0$ implying that $i\not\lra_B j$. 
By approximation lemma (Lemma~\ref{lemma:approximation}) it follows
$i\not\lra_A j$, which is a contradiction. 
\end{proof}

Lemma~\ref{lemma:extendedinfinite} applies to independent random models
as well. More specifically, for an independent random model 
$\{W(k)\}$ and its infinite flow graph $\Ginf$, 
 Lemma~\ref{lemma:extendedinfinite} states: 
if indices $i$ and $j$ are mutually weakly ergodic,
then $i$ and $j$ belong to the same connected component of $\Ginf$. 

As a special consequence of Lemma~\ref{lemma:extendedinfinite},
we obtain that the infinite flow property is necessary for 
the ergodicity (Theorem~1 in~\cite{ErgodicityPaper}).
To see this, we note that by Lemma~\ref{lemma:extendedinfinite},
the ergodic classes of an independent random model $\{W(k)\}$ 
are subsets of the connected components of its infinite flow graph. 
When the model is ergodic,  its infinite flow graph is connected, 
which implies that the model has infinite flow property. 

The converse result of Lemma~\ref{lemma:extendedinfinite} is not true 
in general. For example, let 
$A(k)=\left[\begin{array}{cc}0&1\\1&0\end{array}\right]$ for all $k\geq 0$.
In this case, the infinite flow graph is connected while the model is not 
ergodic. In the resulting dynamics, agents 1 and 2 keep swapping their
initial values $x_1(0)$ and $x_2(0)$.

\section{Ergodicity in class $\M$}\label{sec:classM}
Recall that by Lemma~\ref{lemma:extendedinfinite} the
infinite flow property is necessary for ergodicity of independent 
random models. Also, recall Theorem~\ref{thm:infflowthm} which 
states that the infinite flow property is {\it also 
sufficient} for ergodicity  for independent random models characterized 
by a common steady state $\pi>0$ in expectation and weak feedback property. 

In this section, we introduce a larger class of random models for which 
the infinite flow property is also sufficient for ergodicity. 
This larger class, termed $\M$, encompasses any independent random model 
that can be obtained as an $\ell_1$-approximation of a model with 
a common steady state $\pi>0$ in expectation and weak feedback property.
As a result, the new class includes models that do not necessarily have 
a common steady state $\pi>0$ in expectation. This class and its
properties, which may be of interest in their own right, 
provide important pieces for the development of the main result in 
Section~\ref{sec:infiniteflow}.

Next, we formally introduce the class $\M$ of random models.
\begin{definition}
An independent random model $\{W(k)\}$ belongs to 
the class $\M$ if the dynamic system~\eqref{eqn:dynsys} is such that
for any $t_0\geq 0$ and any $x(t_0)\in \Rm$,
\begin{align}\label{eqn:M2definition}
\sum_{k=t_0}^{\infty}\sum_{i<j}H_{ij}(k)\EXP{(x_i(k)-x_j(k))^2}<\infty,
\end{align}
where $H(k)=\EXP{W^T(k)W(k)}$.  
\end{definition}
The relation in Eq.~\eqref{eqn:M2definition} can be loosely interpreted as
a requirement that the ``total expected variation'' of the entries in $x(k)$
over time is finite, where at each time the variation 
of the entries in $x(k)$ is measured by the quantity
$\sum_{i<j}H_{ij}(k)\EXP{(x_i(k)-x_j(k))^2}$. 
However, relation~\eqref{eqn:M2definition} is more involving than this simple
interpretation since $\sum_{i<j}H_{ij}(k)\EXP{(x_i(k)-x_j(k))^2}$ need
not capture actual variations in $x(k)$, as some of the quantities
$H_{ij}(k)$ may be zero.

Note that, when a model has a common steady state $\pi>0$ in expectation, 
then, almost surely we have (\cite{ErgodicityPaper}, Theorem 5):
\[\sum_{k=0}^{\infty}\sum_{i<j}
\bar H_{ij}(k)(x_i(k)-x_j(k))^2<\infty,\]
where $\bar H(k)=\EXP{W^T(k)\diag(\pi)W(k)}$. Thus,
$\min_{i\in[m]}\pi_i\, \EXP{W^T(k)W(k)}\leq \bar H(k)$, implying that any 
independent random model with a common steady state $\pi>0$ belongs to 
the class $\M$. We observe that the class $\M$ includes chains
that do not necessarily have a common steady state $\pi>0$ in expectation.
To see this, consider the class of deterministic chains 
$\{A(k)\}$ that satisfy a bounded-connectivity condition and have a uniform
lower-bound on their positive entries, such as those discussed 
in~\cite{Tsitsiklis84,Jadbabaie03,Nedic_cdc07,Nedic09,delaypaper}. 
In these models, the sequence 
$d(x(k))=\max_{i\in[m]}x_i(k)-\min_{j\in[m]}x_j(k)$ is (sub)geometric  and, 
thus, it is absolutely summable. Furthermore, 
$H_{ij}(k)=[A^T(k)A(k)]_{ij}\leq m$, which together with relation
$|x_i(k)-x_j(k)|\leq d(x(k))$ for all $i,j\in[m]$, implies that
$\sum_{k=0}^{\infty}\sum_{i<j}[A^T(k)A(k)]_{ij}(x_i(k)-x_j(k))^2
<\infty$, i.e., the defining property of the $\M$-class in 
Eq.~\eqref{eqn:M2definition} holds.

What is interesting is that an $\ell_1$-approximation of a chain
with a common steady state \hbox{$\pi>0$} in expectation
also belongs to $\M$, as shown in the following lemma.

\begin{lemma}\label{lemma:M2approx} 
Let $\{W(k)\}$ be an independent random model with a common steady state 
\hbox{$\pi>0$} in expectation. Let an independent random model $\{U(k)\}$ be  
an $\ell_1$-approximation of $\{W(k)\}$. Then, $\{U(k)\}\in\M$. 
\end{lemma}
\begin{proof}
Define function $V(x)=\sum_{i=1}^{m}\pi_i(x_i-\pi^Tx)^2$.
Let $D=\diag(\pi)$, 
$H(k)=\EXP{U^T(k)U(k)}$ and $L(k)=\EXP{W^T(k)DW(k)}$ for $k\geq 0$. 
Also, let $x(0)\in [0,1]^m$ and $\{x(k)\}$ be the dynamics driven by the chain 
$\{U(k)\}$. Then, we have for any $k\ge0$,
\begin{align}\label{eqn:M2approxOriginal}
\EXP{V(x(k+1))|x(k)}&=\EXP{x^T(k+1)(D-\pi\pi^T)x(k+1)|x(k)}\cr 
&=\EXP{(W(k)x(k)+y(k))^T(D-\pi\pi^T)(W(k)x(k)+y(k))|x(k)}\\
&\le\EXP{V(W(k)x(k))|x(k)}+2\EXP{x(k+1)^T(D-\pi\pi^T)y(k)|x(k)},\nonumber
\end{align}
where $y(k)=(U(k)-W(k))x(k)$ and the last inequality follows by positive
semi-definiteness of the matrix $D-\pi\pi^T$. 
By Theorem~4 in~\cite{ErgodicityPaper}, we also have 
$\EXP{V(W(k)x(k))|x(k)}\leq V(x(k))-\sum_{i<j}L_{ij}(k)(x_i(k)-x_j(k))^2$. 
Therefore, it follows that for all $k\ge0$,
\begin{align}\label{eqn:M2approxSecond}
\EXP{V(x(k+1))|x(k)}&\leq 
V(x(k))-\sum_{i<j}L_{ij}(k)(x_i(k)-x_j(k))^2\cr 
&\qquad +2\EXP{x(k+1)^T(D-\pi\pi^T)y(k)|x(k)}.
\end{align}

Since each $U(k)$ is stochastic, we have $x(k)\in [0,1]^m$ implying that
$\|y(k)\|_{\infty}\leq \|W(k)-U(k)\|_{\infty}$. 
The summation of the absolute values of the entries in the 
$i$th row of $D-\pi\pi^T$ 
is $\pi_i(2-\pi_i)\leq 1$ for all $i$. Thus,
we have $\|(D-\pi\pi^T)y(k)\|_{\infty}\le \|y(k)\|_{\infty} 
\leq \|W(k)-U(k)\|_{\infty}$.
By $x(k)\in[0,1]^m$ for all $k$, we have $\|x(k)\|_1\le m$
impying that $x(k+1)^T(D-\pi\pi^T)y(k)\leq m\|W(k)-U(k)\|_{\infty}$. 
Since $U(k)$ and $W(k)$ are independent of $x(k)$, we obtain
\[\EXP{x(k+1)^T(D-\pi\pi^T)y(k)|x(k)}\le m\EXP{\|W(k)-U(k)\|_{\infty}},\]
which when combined with Eq.~(\ref{eqn:M2approxSecond}) yields
\begin{align*}
\EXP{V(x(k+1))|x(k)}\leq 
V(x(k))-\sum_{i<j}L_{ij}(k)(x_i(k)-x_j(k))^2+2m\EXP{\|W(k)-U(k)\|_{\infty}}.
\end{align*}
Note that  
$\sum_{k=0}^{\infty}\EXP{\|W(k)-U(k)\|_{\infty}}<\infty$ since
$\{U(k)\}$ is an $\ell_1$-approximation of $\{W(k)\}$. Thus, 
by the Robbins-Siegmund theorem (\cite{Poznyak}, page 164), 
$\sum_{k=0}^{\infty}\sum_{i<j}L_{ij}(k)(x_i(k)-x_j(k))^2{<\infty}$ 
almost surely.

The last step is to show that the difference between the two sums 
$\sum_{k=0}^{\infty}\sum_{i<j}L_{ij}(k)(x_i(k)-x_j(k))^2$ and 
$\sum_{k=0}^{\infty}\sum_{i<j}H_{ij}(k)(x_i(k)-x_j(k))^2$ is finite. 
Since $|W_{\ell i}(k)-U_{\ell i}(k)|\leq \|W(k)-U(k)\|_{\infty}$ for any 
$i,\ell\in[m]$, using the definitions of $H(k)$ and $L(k)$, we obtain
\begin{align}\label{eqn:HLcompare}
\pi_{\min}H_{ij}(k)
&\leq\sum_{\ell=1}^m\EXP{\pi_{\ell}U_{\ell i}(k)U_{\ell j}(k)}\cr 
&= \sum_{\ell=1}^m\EXP{\pi_{\ell}(W_{\ell i}(k)
+[W(k)-U(k)]_{\ell i})(W_{\ell j}(k)+[W(k)-U(k)]_{\ell j})}\cr 
&\leq L_{ij}(k)+\EXP{\|W(k)-U(k)\|_{\infty}
\sum_{\ell=1}^m\pi_{\ell}(W_{{\ell i}}(k)+W_{{\ell j}}(k)+1)},
\end{align}
where $\pi_{\min}=\min_{i\in[m]}\pi_i$. The last inequality is obtained using 
the triangle inequality and the following 
\[[W(k)-U(k)]_{\ell i}[W(k)-U(k)]_{\ell j}
\leq \|W(k)-U(k)\|^2_{\infty}\leq \|W(k)-U(k)\|_{\infty}.\]
Therefore, using relation \eqref{eqn:HLcompare},
$W_{\ell j}(k)\in[0,1]$, and the stochasticity of $\pi$, we have
\begin{align}\nonumber 
\pi_{\min}H_{ij}(k)\leq L_{ij}(k)+3\EXP{\|W(k)-U(k)\|_{\infty}}.
\end{align}
Therefore,
\begin{align}\nonumber
\pi_{\min}\sum_{k=0}^{\infty}\sum_{i<j}&H_{ij}(k)(x_i(k)-x_j(k))^2\cr 
&\leq \sum_{k=0}^{\infty}\sum_{i<j}L_{ij}(k)(x_i(k)-x_j(k))^2+
3\sum_{k=0}^{\infty}\sum_{i<j}\EXP{\|W(k)-U(k)\|_{\infty}}(x_i(k)-x_j(k))^2\cr 
&\leq \sum_{k=0}^{\infty}\sum_{i<j}L_{ij}(k)(x_i(k)-x_j(k))^2+3m^2
\sum_{k=0}^{\infty}\EXP{\|W(k)-U(k)\|_{\infty}},
\end{align}
where the last inequality holds since $(x_i(k)-x_j(k))^2\leq 1$.  
Since $\sum_{k=0}^{\infty}\EXP{\|W(k)-U(k)\|_{\infty}}<\infty$ and 
$\pi_{\min}>0$, and we have 
shown that $\sum_{k=0}^{\infty}\sum_{i<j}L_{ij}(k)(x_i(k)-x_j(k))^2<\infty$ 
almost surely, it follows that 
$\sum_{k=0}^{\infty}\sum_{i<j}H_{ij}(k)(x_i(k)-x_j(k))^2<\infty$ almost surely.
\end{proof}

We next show that, for the chains in class $\M$ that have 
weak feedback property, the infinite flow is also sufficient for ergodicity.
This result non-trivially extends the class of independent chains 
to which the infinite flow theorem applies. 

\begin{theorem}\label{thm:sufficient}
Let $\{W(k)\}$ be a class $\M$ model with weak feedback property. 
Then, the infinite flow property is both necessary and sufficient 
for the ergodicity of the model. 
\end{theorem}
\begin{proof}
The necessity of the infinite flow property follows by 
Lemma~\ref{lemma:extendedinfinite}. For the converse, assume that the model 
has the infinite flow property. Let $t_0=0$ and let $x(0)\in\mathbb{R}^m$
be arbitrary. Then, since $\{W(k)\}$ is in the $\M$-class, it follows
\[\sum_{k=0}^{\infty}\sum_{i<j}H_{ij}(k)\EXP{(x_i(k)-x_j(k))^2}<\infty,\] 
where $H(k)=\EXP{W^T(k)W(k)}$. Due to the weak feedback property, we have 
$\EXP{W^i(k)^TW^j(k)}\geq \gamma \EXP{W_{ij}(k)+W_{ji}(k)}$ for 
all $k\ge0$, all $i,j\in[m]$ and some $\gamma>0$. 
Using the independence of the model and relation 
$H_{ij}(k)=\EXP{W^i(k)^TW^j(k)}$, we obtain
\begin{align*}
&\sum_{k=0}^{\infty}\sum_{i<j}\EXP{(W_{ij}(k)+W_{ji}(k))(x_i(k)-x_j(k))^2}\cr 
&= \EXP{\sum_{k=0}^{\infty}\sum_{i<j}(W_{ij}(k)+W_{ji}(k))(x_i(k)-x_j(k))^2} 
<\infty,
\end{align*}
where the equality holds by 
$(W_{ij}(k)+W_{ji}(k))(k)(x_i(k)-x_j(k))^2\ge0$ and the monotone convergence 
theorem (\cite{Folland}, page 50). Consequently,
\[\sum_{k=0}^{\infty}
\sum_{i<j}(W_{ij}(k)+W_{ji}(k))(k)(x_i(k)-x_j(k))^2<\infty\qquad 
\mbox{almost surely}.\] 
Since the model has the infinite flow property, 
by Lemma~3 in \cite{CDC2010}, we have 
$\lim_{k\rightarrow\infty}(x_{\max}(k)-x_{\min}(k))=0$ almost surely for  
any $x(0)\in\re^m$, implying that the system reaches consensus 
almost surely. Since 
$\sum_{k=t}^{\infty}\sum_{i<j}H_{ij}(k)\EXP{(x_i(k)-x_j(k))^2}<\infty$ for any 
starting time $t_0\geq 0$, by the same argument it follows that the model is 
ergodic. 
\end{proof}

Theorem~\ref{thm:sufficient} implies that the domain of the infinite 
flow theorem (Theorem~\ref{thm:infflowthm}) can be extended for
the chains with weak feedback property to a larger class of 
the $\M$ models that do not require the existence of a common
steady state vector $\pi>0$ in expectation. 
This is done in the following theorem.
\begin{theorem}\label{thm:infflowthmM2}
Theorem~\ref{thm:infflowthm} applies to any random model that belongs to 
the class $\M$ and has weak feedback property.
\end{theorem} 
\begin{proof}
By Theorem~\ref{thm:sufficient}, (d) implies (a). 
The implications (a) $\Rightarrow$ (b) $\Rightarrow$ (c) $\Rightarrow$ 
(d) are true for any independent random model as proven 
in~\cite{ErgodicityPaper} (Theorem~7).
\end{proof}
With this theorem we conclude our discussion on ergodic models.
In the following section, we shift our focus on the models that are not 
necessarily ergodic.
\section{Ergodicity Classes and Infinite Flow Graph}\label{sec:infiniteflow}
In this section, we study models with a common steady state 
$\pi>0$ in expectation and weak feedback property that are not necessarily 
ergodic, which is equivalent to \textit{not having infinite flow property}. 
Our goal is to investigate the limiting behavior of 
such models and, in particular, to characterize their ergodicity classes. 
We do this by considering the infinite flow graph of a model. To illustrate
what our goal is,  recall the gossip model of Section~\ref{sec:gossip}. 
In Theorem~\ref{thm:gossip}, 
we showed that the ergodicity classes of the model are related to
the connected components of the infinite flow graph of the model.
Here, we show that the same result holds for a more general independent random
model. In the process, we use {diagonal approximation} of the model.

In particular, consider an independent random model  $\W$ 
with a common steady state $\pi>0$ in expectation and weak feedback property. 
Let $\Ginf$ be the infinite flow graph of the model. Assume that 
$\Ginf$ has $\tau\geq 1$ connected components, and let 
$S_1,\ldots,S_\tau\subset [m]$ be the sets of vertices of the connected 
components in $\Ginf$. Let 
$S_1=\{1,\ldots,a_1\}$, $S_2=\{a_1+1,\ldots,a_2\},\ldots,$ 
$S_\tau=\{a_{\tau-1}+1,\ldots,a_{\tau}=m\}$ for 
$1\leq a_1\leq\ldots\leq a_{\tau}=m$, and let $m_r=|S_r|=a_r-a_{r-1}$ 
be the number of vertices in the $r$th component, where $a_0=0$. 
Using the connected components of $\Ginf$, we define
 the \textit{diagonal approximation} 
$\{\tilde{W}(k)\}$ of $\W$, as follows.

\begin{definition}(\textit{Diagonal Approximation})
Let $\W$ be a random model.
For $1\leq r\leq \tau$, let the random model $\{W^{(r)}(k)\}$ 
in $\R^{m_r}$ be given as follows: for $i,j\in [m_r]$,
\begin{align}\label{eqn:diagonalmatrix}
W^{(r)}_{ij}(k)=
\left\{
\begin{array}{ll}
W_{(i+a_{r-1})(i+a_{r-1})}(k)+\sum_{\ell\in \bar{S_r}}W_{(i+a_{r-1})\ell}(k)
& \mbox{if $j=i$,}\\
W_{(i+a_{r-1})(j+a_{r-1})}(k)& \mbox{if $j\not=i$.}
\end{array}\right.
\end{align}
The diagonal approximation of the model $\W$ is 
the model $\{\tilde{W}(k)\}$ defined by 
\begin{align*}
\tilde{W}(k)=\diag(W^{(1)}(k),\ldots,W^{(\tau)}(k))=\left(
\begin{array}{cccc}
W^{(1)}(k)&0&\cdots&0\\
0&W^{(2)}(k)&\cdots&0\\
\vdots&\vdots&\ddots&\vdots\\
0&0&\ldots&W^{(\tau)}(k)
\end{array}\right).
\end{align*}
\end{definition}

Basically, in the diagonal approximation, the links between the connected 
components are removed. At the same time, 
in order to preserve the stochasticity of the matrices, 
the weights of  the removed links are added to the self-feedback weight
of corresponding agents.

We now present Lemma~\ref{lemma:diagonalapprox} which 
provides basic properties of diagonal approximation. 
The lemma shows that the coupling between the
connected components is weak enough to guarantee that
diagonal approximation is an $\ell_1$-approximation. 
At the same time, the coupling within
each of the diagonal submodels is rather strong, 
as each submodel possesses infinite flow property.

\begin{lemma}\label{lemma:diagonalapprox}
Let $\{W(k)\}$ be an independent random model and $\{\tilde{W}(k)\}$ be 
its diagonal approximation. Then, $\{\tilde{W}(k)\}$ is an 
$\ell_1$-approximation of $\{W(k)\}$. Furthermore, for every $r=1,\ldots,\tau$,
the random model $\{W^{(r)}(k)\}$ as given in Eq.~\eqref{eqn:diagonalmatrix} 
has infinite flow property. 
\end{lemma}
\begin{proof}
First we show that $\tilde{W}(k)$ is a stochastic matrix for any $k\geq 0$. 
Due to the diagonal structure of the matrix $\tilde{W}(k)$ 
(Eq.~(\ref{eqn:diagonalmatrix})), it suffices to show that $W^{(r)}(k)$ is 
stochastic for $1\leq r\leq \tau$. By the definition of 
$W^{(r)}(k)$, we have $W^{(r)}(k)\geq 0$. Also, for any $i\in [m_r]$, we have
\begin{align*}
\sum_{j=1}^{m_r}W^{(r)}_{ij}(k)&=W^{(r)}_{ii}(k)
+\sum_{j\not=i,j\in[m_r]}W^{(r)}_{ij}(k)\cr 
&=W_{(i+a_{r-1})(i+a_{r-1})}(k)+\sum_{\ell\in \bar{S_v}}W_{(i+a_{r-1})\ell}(k)+
\sum_{\ell\not=i+a_{r-1},\ell\in S_r}W_{(i+a_{r-1})\ell}(k)\cr 
&=\sum_{\ell=1}^{m}W_{(i+a_r)\ell}(k)=1.
\end{align*}

Now, let $i\in S_r$ for $1\leq r\leq \tau$. Then, for any $j\not=i$, we have 
two cases:\\
(i) If $j\in S_r$, by the definition of $\tilde{W}(k)$, then
$W_{ij}(k)=\tilde{W}_{ij}(k)$. Hence, $|W_{ij}(k)-\tilde{W}_{ij}(k)|=0$. \\
(ii) If $j\not\in S_r$, then $\tilde{W}_{ij}(k)=0$ and, thus, 
$|W_{ij}(k)-\tilde{W}_{ij}(k)|=W_{ij}(k)$. \\
For $j=i$, we have $\tilde{W}_{ii}(k)
=W_{ii}(k)+\sum_{j\not\in S_r}W_{ij}(k)$. Hence, $|W_{ij}(k)-\tilde{W}_{ij}(k)|
=\sum_{j\not\in S_r}W_{ij}(k)$, implying that
$\sum_{j=1}^{m}|W_{ij}(k)-\tilde{W}_{ij}(k)|=2\sum_{j\not\in S_r}W_{ij}(k).$
By summing these relations over all $i\in S_r$, we obtain
for all $r=1,\ldots,\tau$,
\[\sum_{i\in S_r}\sum_{j=1}^{m}|W_{ij}(k)-\tilde{W}_{ij}(k)|
=2\sum_{i\in S_r}\sum_{j\not\in S_r}W_{ij}(k)\leq 2W_{S_r}(k).\]
Again, by summing the preceding inequalities  over $r=1,\ldots,\tau$, we 
further obtain
\[\sum_{r=1}^{\tau}\sum_{i\in S_r}\sum_{j=1}^{m}|W_{ij}(k)-\tilde{W}_{ij}(k)|
\leq 2\sum_{r=1}^{\tau}W_{S_r}(k).\]
Note that $\sum_{r=1}^{\tau}\sum_{i\in S_r}
\sum_{j=1}^{m}|W_{ij}(k)-\tilde{W}_{ij}(k)|
=\sum_{i,j\in[m]}|W_{ij}(k)-\tilde{W}_{ij}(k)|$. Since
$S_1,\ldots,S_\tau$ are the sets of vertices of the connected components of 
$\Ginf$, it follows that $\sum_{k=0}^{\infty}\sum_{r=1}^{\tau}W_{S_r}(k)
<\infty$ almost surely. Therefore, by combining the above facts, we conclude
that
\[\sum_{k=0}^{\infty}\sum_{i,j\in[m]}^{m}|W_{ij}(k)-\tilde{W}_{ij}(k)|
<\infty \quad a.s.,\]
which proves that $\{\tilde{W}(k)\}$ is an $\ell_1$-approximation of 
$\{W(k)\}$. 

To prove that each submodel $\{W^{(r)}(k)\}$ has infinite flow property, let 
$V\subset S_r$ be nonempty but arbitrary. Since $S_r$ is the set of 
vertices of the $r$th connected component of $\Ginf$, there is
an edge $\{i,j\}\in\Einf$ such that $i\in V$ and $j\in \bar{V}$. 
By the definition of $W^{(r)}(k)$, for $i_r=i-a_{r-1}$ and $j_r=j-a_{r-1}$, 
we have $W^{(r)}_{i_rj_r}(k)+W^{(r)}_{j_ri_r}(k)=W_{ij}(k)+W_{ji}(k)$. 
Since $\{i,j\}\in\Einf$, it follows
\[\sum_{k=0}^{\infty}(W^{(r)}_{i_rj_r}(k)+W^{(r)}_{j_ri_r}(k))=
\sum_{k=0}^{\infty}(W_{ij}+W_{ji})=\infty,\]
thus showing that the infinite flow graph of $\{W^{(r)}(k)\}$ is connected. 
Hence, $\{W^{(r)}(k)\}$ has infinite flow property. 
\end{proof}

Lemma~\ref{lemma:diagonalapprox}, together with approximation lemma 
(Lemma~\ref{lemma:approximation}), provides us with basic tools
for our study on the relations between the ergodicity classes and the
infinite flow graph. Having these lemmas, we are now ready to characterize 
these relations for a certain random models. 

Recall that in Lemma~\ref{lemma:extendedinfinite}, we showed that 
if $i$ and $j$ are mutually weakly ergodic, then $i$ and $j$ belong to 
the same connected component of the infinite flow graph. 
Since mutual ergodicity is more restrictive than mutual weak ergodicity, 
Lemma ~\ref{lemma:extendedinfinite} is valid
when $i$ and~$j$ are mutually ergodic. The existence of models for which 
the converse statement holds is ensured by Theorem~\ref{thm:gossip}, which 
shows that the extended gossip model is one of them.
The following theorem provides a characterization of these models 
in a more general setting than the gossip.

\begin{theorem}\label{thrm:infiniteflowgraph}
Let $\{W(k)\}$ be an independent model with a common steady state 
$\pi>0$ in expectation and weak feedback property. Then, $i\lrA j$ 
if and only if $i$ and $j$ are in the same connected component of 
the infinite flow graph $\Ginf$ of the model. 
\end{theorem}
\begin{proof}
Since mutual ergodicity implies mutual weak ergodicity,
the ``if'' part follows from Lemma~\ref{lemma:extendedinfinite}.
To show the ``only if'' part, we use two $\ell_1$-approximations successively. 
We proceed through the following steps to prove this result. First, in order to
have weak feedback property, we construct an $\ell_1$-approximation of 
the diagonal approximation of the model $\W$. Next, we prove that the
resulting chain has weak feedback property. Finally, complete the proof 
by making use of the results developed  in the preceding sections.

\noindent
\underline{Approximation}:
Let $\Ginf$ have $\tau$ connected components, and let $S_1,\ldots,S_\tau$
be the vertex sets corresponding to the connected components of $\Ginf$.
Let $\pi_{\min}=\min_{i\in[m]}\,\pi_i>0$. 
Consider the diagonal approximation $\{\tw(k)\}$ of $\{W(k)\}$ with 
$\tw^{(r)}(k)$ defined as in Eq.~\eqref{eqn:diagonalmatrix} 
for $r\in[\tau]$. Let $M(k)=\EXP{\max_{i,j\in[m]}|\tw_{ij}(k)-W_{ij}(k)|}$. 
Since $\{\tw(k)\}$ is an $\ell_1$-approximation of $\{W(k)\}$, we have  
$\sum_{k=0}^{\infty}M(k)<\infty$, implying that 
$\lim_{k\rightarrow\infty}M(k)=0$. Thus, there exists $N\geq 0$ 
such that $M(k)\leq \frac{\pi_{\min}}{8m^2}$ for all $k\geq N$. 
Now, let $U(k)=I$ for $k<N$ and for $k\ge N$,
\[U^{(r)}(k)=(1-d(k))\tw^{(r)}(k)+\frac{d(k)}{m_r}e^{(m_r)}e^{(m_r)T},\] 
where $d(k)=\frac{4m^2}{\pi_{\min}}M(k)$ for $k\geq 0$ and 
$e^{(m_r)}\in\R^{m_r}$ is the vector with all entries equal to 1. 
Note that $\frac{1}{m_r}e^{(m_r)}e^{(m_r)T}$ is a stochastic matrix.
Since $M(k)\leq \frac{\pi_{\min}}{8m^2}$ for $k\geq N$, we have 
$d(k)\in[0,\frac{1}{2}]$, thus implying that $U^{(r)}(k)$ is 
a convex combination of stochastic matrices and, hence, stochastic. 

Since $\sum_{k=0}^{\infty}d(k)<\infty$ it follows 
$\sum_{k=0}^{\infty}M(k)<\infty$, thus implying that the model 
$\{U^{(r)}(k)\}_{k\geq N}$ is an $\ell_1$-approximation of 
$\{\tw^{(r)}(k)\}_{k\geq N}$. 
Since the entries of each matrix $\tw^{(r)}(k)$ are in $[0,1]$, 
changing finitely many matrices in a chain cannot change infinite flow 
properties. Thus, $\{U^{(r)}(k)\}$ is an $\ell_1$-approximation of 
$\{\tw^{(r)}(k)\}$ and the model $\{U(k)\}$ with matrices defined by 
$U(k)=\diag(U^{(1)}(k),\ldots,U^{(\tau)}(k))$, $k\geq 0$, is an 
$\ell_1$-approximation of $\{\tw(k)\}$. By Lemma~\ref{lemma:diagonalapprox},
$\{\tw(k)\}$ is an $\ell_1$-approximation of the original model $\{W(k)\}$ 
and, therefore, $\{U(k)\}$ is an $\ell_1$-approximation of $\{W(k)\}$. 

\noindent
\underline{Weak Feedback}: Now, we show that the model $\{U(k)\}$ has weak 
feedback property with feedback coefficient 
$\zeta=\frac{1}{2}\min(\gamma,\frac{\pi_{\min}}{4m})$. For $k<N$, we have 
$U(k)=I$ which has weak feedback property with coefficient~$1$. So we
consider $U(k)$ for an arbitrary $k\geq N$. Let $r\in[\tau]$ 
be arbitrary and let $Q=U^{(r)}(k)$ to keep notation simple. 
For $i,j\in [m_r]$ recall that their corresponding indices in $[m]$ are given 
by $i_r=i+a_{r}-1$, $j_r=j+a_{r}-1$. Also, recall that $W^s$ denotes the $s$th 
column vector of a matrix $W$. Using this, for any $i,j\in [m_r]$
with $i\not=j$ we have:
\begin{align}\label{eqn:feedbackstartpoint}
{Q^i}^TQ^j&=
\left((1-d(k))\tw^{(r)i}(k)+\frac{d(k)}{m_r}e^{(m_r)}\right)^T 
\left((1-d(k))\tw^{(r)j}(k)+\frac{d(k)}{m_r}e^{(m_r)}\right)\cr 
&\geq (1-d(k))^2(\tw^{(r)i}(k))^T\tw^{(r)j}(k)
+\frac{(1-d(k))d(k)}{m}e^{(m_r)T}(\tw^{(r)i}(k)+\tw^{(r)j}(k)).\qquad
\end{align}
where in the last inequality we use $1-d(k)\ge0$ and
$\frac{1}{m_r}\ge \frac{1}{m}$.
Since $e^{(m_r)T}\tw^{(r)i}(k)=e^T\tw^{i_r}(k)$ it follows that
\begin{align*}
\frac{1}{m}e^{(m_r)T}\tw^{(r)i}(k)
&=\frac{1}{m}e^TW^{i_r}(k) + \frac{1}{m}e^T(\tw^{i_r}(k)-W^{i_r}(k))\cr
&\geq 
\frac{1}{m}e^TW^{i_r}(k) - \max_{i',j'\in[m]}|\tw_{i'j'}(k)-W_{i'j'}(k)|,
\end{align*}
where the inequality holds by the stochasticity of $\frac{1}{m}e$. Therefore, 
\[\EXP{e^{(m_r)T}\tw^{(r)i}(k)} \geq
\EXP{\frac{1}{m}e^TW^{i_r}(k)} 
- \EXP{\max_{i,'j'\in[m]}|\tw_{i'j'}(k)-W_{i'j'}(k)|}=\frac{1}{m}\pi_i-M(k),\]
where the equality follows from $\pi$ being a common steady state 
in expectation for $\{W(k)\}$ and the definition of $M(k)$. Similarly, we have 
$\EXP{\frac{1}{m}e^{(m_r)T}\tw^{(r)j}(k)}\geq\frac{1}{m}\pi_j-M(k)$. 
Taking the expectation of the both sides in 
Eq.~\eqref{eqn:feedbackstartpoint} and using the preceding inequalities,
we obtain
\begin{align}\label{eqn:feedback2}
\EXP{Q^{iT}Q^j}&\geq (1-d(k))^2 \EXP{(\tw^{i_r}(k))^T\tw^{j_r}(k)}
+(1-d(k))d(k)\left(\frac{\pi_i+\pi_j}{m}-2M(k)\right)\cr 
&\geq (1-d(k))^2
\EXP{(\tw^{i_r}(k))^T\tw^{j_r}(k)}+(1-d(k))d(k)\frac{\pi_{\min}}{m},
\end{align}
which holds by $M(k)\leq\frac{\pi_{\min}}{8m^2}\leq \frac{\pi_{\min}}{m}$ 
for $k\geq N$ and $\pi_i,\pi_j\geq \pi_{\min}$ for any $i,j\in[m_r]$. 

Since $\tw_{\ell i_r}(k)\geq W_{\ell i_r}(k)-\max_{ij}|\tw_{ij}(k)-W_{ij}(k)|$
for all $\ell\in[m]$, it follows that
\[\EXP{(\tilde{W}^{i_r}(k))^T\tilde{W}^{j_r}(k)}
\geq \EXP{(W^{i_r}(k))^TW^{j_r}(k)}-2mM(k)
= \EXP{(W^{i_r}(k))^TW^{j_r}(k)}-d(k)\frac{\pi_{\min}}{2m},\]
where in the last equality we use $2mM(k)=d(k)\frac{\pi_{\min}}{2m}$,
which follows from $\frac{4m^2}{\pi_{\min}}M(k)=d(k)$.
Using the above relation in Eq.~(\ref{eqn:feedback2}), we have
\begin{align*}
\EXP{Q^{iT}Q^j}&\geq (1-d(k))^2\EXP{(W^{i_r}(k))^TW^{j_r}(k)}
-(1-d(k))^2 d(k)\frac{\pi_{\min}}{2m}+(1-d(k))d(k)\,\frac{\pi_{\min}}{m}. 
\end{align*}
Since $(1-d(k))^2\leq 1-d(k)$ it follows 
\begin{align}\label{eqn:feedback4}
\EXP{Q^{iT}Q^j}&\geq (1-d(k))^2\EXP{(W^{i_r}(k))^TW^{j_r}(k)} 
+(1-d(k))d(k)\,\frac{\pi_{\min}}{2m}\cr 
&\geq (1-d(k))^2\gamma\EXP{W_{i_rj_r}(k)+W_{j_ri_r}(k)}+(1-d(k))d(k)\,
\frac{\pi_{\min}}{2m},
\end{align}
where the last inequality follows by weak feedback property of $\{W(k)\}$.

 Since $i_r,j_r\in S_r$ and 
$i_r\not=j_r$, by the construction of $\tw(k)$, we have 
$\tw_{i_rj_r}(k)=W_{i_rj_r}(k)$. Hence, 
$\EXP{Q_{ij}+Q_{ji}}=(1-d(k))(\EXP{W_{i_rj_r}(k)}
+\EXP{W_{j_ri_r}(k)})+\frac{2d(k)}{m_r}$. 
By combining this with Eq.~(\ref{eqn:feedback4}), we have 
\begin{align*}
\EXP{Q^{iT}Q^j}&\geq 
(1-d(k))\gamma \left(\EXP{Q_{ij}+Q_{ji}}-\frac{2d(k)}{m_r}\right)
+ (1-d(k))d(k)\, \frac{\pi_{\min}}{2m}\cr 
&=(1-d(k))\gamma \EXP{Q_{ij}+Q_{ji}}+(1-d(k))d(k)
\left(-\frac{2\gamma}{m_r}+\frac{\pi_{\min}}{2m}\right).
\end{align*}
Without loss of generality, we may assume $\gamma\leq \frac{\pi_{\min}}{4m}$, 
otherwise, $\gamma'=\frac{\pi_{\min}}{4m}<\gamma$ would be a feedback 
coefficient for $\{W(k)\}$ and the arguments hold for $\gamma'$. But for 
$\gamma\leq \frac{\pi_{\min}}{4m}$, we have $-\frac{2\gamma}{m_r}
+\frac{\pi_{\min}}{2m}\geq 0$. Thus, 
\begin{align*}
\EXP{Q^{iT}Q^j}\geq \frac{\gamma}{2}\,\EXP{Q_{ij}+Q_{ji}},
\end{align*}
which follows from $d(k)\leq \frac{1}{2}$. 
Note that we defined $Q=U^{(r)}(k)$ where $r\in[\tau]$ and $k\geq N$ 
was arbitrary.
Hence, each of the decoupled random models $\{U^{(r)}(k)\}$ has weak feedback 
property with feedback constant $\zeta=\frac{1}{2}\min(\gamma,\frac{\pi_{\min}}{4m})$. 

\noindent
\underline{Last Step}:
Let $x(0)\in\Rm$ and $\{x(k)\}$ be dynamics resulting from
the chain $\{U(k)\}$ according to Eq.~(\ref{eqn:dynsys}). 
By Lemma~\ref{lemma:M2approx} $\{U(k)\}\in\M$, implying that 
$\sum_{k=0}^{\infty}\sum_{i<j}L_{ij}(k)\EXP{(x_i(k)-x_j(k))^2}<\infty$ 
almost surely, where $L(k)=\EXP{U^T(k)U(k)}$. Hence, for any $r\in[\tau]$, 
\[\sum_{k=0}^{\infty}\sum_{{i<j}\atop{i,j\in S_r}}
L_{ij}(k)\EXP{(x_i(k)-x_j(k))^2}<\infty.\]
Due to the diagonal structure of $U(k)$, we have:\\
\noindent
(a) \ $x(k)=(x^{(1)}(k),\cdots,x^{(\tau)}(k))$, where 
$x^{(r)}_i(0)=x_{i_r}(0)$, so $\{x^{(r)}(k)\}$ are the sequences of 
random vectors in $\mathbb{R}^{m_r}$ driven by the individual chains 
$\{U^{(r)}(k)\}$.\\
\noindent
(b) \ For any $i,j\in [m_r]$ and any $r\in[\tau]$, recalling that
$\ell_r=\ell+a_{r-1}$
\begin{align*}
L_{i_rj_r}(k)
=\EXP{\sum_{\ell\in[m]}U_{\ell i_r}(k)U_{\ell j_r}(k)}
=\EXP{\sum_{\ell\in S_r}
U_{\ell i_r}(k)U_{\ell j_r}(k)}
=\EXP{\sum_{\bar{\ell}\in [m_r]}
U^{(r)}_{\bar \ell i}(k)U^{(r)}_{\bar\ell j}(k)}.
\end{align*}
In view of the above observations, the random dynamics in $\Rm$ 
induced by  $\{U(k)\}\in\M$ decomposes into $\tau$ random dynamics 
in $\mathbb{R}^{m_1},\ldots,\mathbb{R}^{m_\tau}$ induced by 
$\{U^{(1)}(k)\},\ldots,\{U^{(\tau)}(k)\}$ all of which belong to class $\M$. 
Each model $\{\tw^{(r)}(k)\}$ has infinite flow property so 
its $\ell_1$-approximation also has infinite flow property. Furthermore, 
as we showed, each model $\{U^{(r)}(k)\}$ has weak feedback property.
Hence, by Theorem~\ref{thm:infflowthmM2}, $\{U^{(r)}(k)\}$ is an ergodic 
chain for any $r\in[\tau]$, which implies that $i\lrA j$ for any 
$i,j\in [m_r]$. Moreover, since $U(k)=\diag(U^{(1)}(k),\ldots,U^{(\tau)}(k))$, 
it follows $i_r\lrA j_r$ in $\{U(k)\}$. 
By approximation lemma (Lemma~\ref{lemma:approximation}), we have 
$i'\lrA j'$ in the original chain $\{W(k)\}$ if and only if 
$i'\lrA j'$ in $\{U(k)\}$, which is in turn true (by the structure of $U(k)$)
if and only if $i',j'\in S_r$ for some $r\in[\tau]$. 
\end{proof}

Theorem~\ref{thrm:infiniteflowgraph} implies that 
the model satisfying the conditions of the theorem is stable.
Furthermore, it shows that the ergodicity classes of such a model
can be fully characterized by considering the connected components in the
infinite flow graph of the model.

Our next result further strengthens Theorem~\ref{thrm:infiniteflowgraph}
by showing that this theorem also applies to an $\ell_1$-approximation of 
a model that satisfies the conditions of the theorem. Furthermore, 
such an $\ell_1$-approximation and its expected model have the same ergodicity 
classes. 

\begin{theorem}\label{thrm:extendedinfflow}
(\textit{Extended Infinite Flow Theorem}) Let an independent random
model $\W$ be an $\ell_1$-approximation of an independent random model 
with a common steady state $\pi>0$ in expectation and weak feedback property. 
Let $\Ginf$ be the infinite flow graph of $\W$ and $\bar{G}^{\infty}$ be 
the infinite flow graph of the expected model $\{\bar W(k)\}$, where 
$\bar W(k)=\EXP{W(k)}$. 
Then, $\W$ is stable almost surely and the following statements are equivalent:
\begin{itemize}
\item[(a)] $i\lrA_W j$.
\item[(b)] $i\lrA_{\bar{W}} j$.
\item[(c)] $i$ and $j$ belong to the same connected component of 
$\bar{G}^{\infty}$.
\item[(d)] $i$ and $j$ belong to the same connected component of $\Ginf$.
\end{itemize}
\end{theorem}
\begin{proof}
Since $\W$ is independent, (a) implies (b) by the dominated convergence
theorem (\cite{Durrett} page 15). By Lemma~\ref{lemma:extendedinfinite}, 
(b) implies (c). Since $0\leq W_{ij}(k)\leq 1$ and the model is independent, 
by Kolmogorov's three series theorem ([7], page 63), 
$\sum_{k=0}^{\infty}W_{ij}(k)<\infty$ holds a.s. only if 
$\sum_{k=0}^{\infty}\EXP{W_{ij}(k)}<\infty$, so (c) implies (d). 
Finally, by Theorem~\ref{thrm:infiniteflowgraph},
(d) and (a) are equivalent.
\end{proof}

By Theorem~\ref{thrm:extendedinfflow}, we have that any dynamics driven
by a random model satisfying the assumptions of the theorem 
converges almost surely. This, however, need not be true 
if either $\pi>0$ or week feedback assumption 
of the theorem is violated, as seen in the following examples. 

\begin{example}
Let matrices $W(k)$ be given by
\[W(k)=\left[\begin{array}{ccc}
1&0&0\\
u_1(k)&u_2(k)&u_3(k)\\
0&0&1
\end{array}\right],\] 
where $u(k)=(u_1(k),u_2(k),u_3(k))^T$ are i.i.d.\ random 
vectors distributed uniformly in the probability simplex of $\R^3$. Then, 
starting from the point $x(0)=(0,\frac{1}{2},1)^T$, the dynamics will not 
converge. This model has infinite flow property and satisfies all assumptions 
of Theorem~\ref{thrm:extendedinfflow} except for the assumption $\pi>0$. 
\end{example}

\begin{example}
Consider the random permutation model. Specifically, let $W(k)$ 
be the i.i.d.\ model with $W(k)$ randomly and uniformly chosen from the set 
of permutation matrices in $\Rm$. 
Starting from any initial point, this model just permutes the coordinates of 
the initial point. Therefore, the dynamic is not converging for any 
$x(0)$ that lies outside the subspace spanned by the vector $e$.
The model has infinite flow property and has the common steady state 
$\pi=\frac{1}{m}e$ in expectation. However, the model does not have 
weak feedback property, since $\EXP{W^i(k)^TW^j(k)}=0$ for $i\ne j$ while 
$\EXP{W_{ij}(k)}+\EXP{W_{ji}(k)}>0$. 
\end{example}

\section{Applications}\label{sec:applications}
Here, we consider some applications of 
Theorem~\ref{thm:infflowthm} and its extended variant to ergodicity classes
in Theorem~\ref{thrm:extendedinfflow}. 
First, we discuss the broadcast-gossip model for a time-varying network 
and, then, we consider a link failure process on 
random networks. 

\subsection{Broadcast Gossip Algorithm on Time-Changing Networks}
\label{sec:broadcast}
Broadcast gossip algorithm has been presented and analyzed 
in~\cite{Aysal08,Aysal09} for consensus over a static network. Here, 
we propose broadcast gossip algorithm for time-varying networks and provide a
necessary and sufficient condition for ergodicity. 
Suppose that we have a network with $m$ nodes and a sequence of simple 
undirected graphs $\{G(k)\}$, where $G(k)=([m],\E(k))$ and  $\E(k)$ 
represents the topology of the network at time $k$. The sequence
$\{G(k)\}$  is assumed to be deterministic. 
Suppose that at time $k$, agent $i\in[m]$ wakes up with probability 
$\frac{1}{m}$ (independently of the past) and 
broadcasts its value to its neighboring agents 
$N_i(k)=\{j\in[m]\mid \{i,j\}\in \E(k)\}$. At this time, each agent 
$j\in N_i(k)$ updates its estimate as follows:
\begin{align}\nonumber
x_j(k+1)=\g(k)x_i(k)+(1-\g(k))x_j(k),
\end{align}
where $\g(k)\in(0,\gamma]$ is a mixing parameter of the system at time $k$ 
and $\gamma\in (0,1)$. The other agents keep their values unchanged, i.e., 
$x_j(k+1)=x_j(k)$ for $j\not\in N_i(k)$.
Therefore, in this case the vector $x(k)$ of agents' estimates $x_i(k)$
evolves in time according to~\eqref{eqn:dynsys} where
\begin{align}\label{eqn:broadcastgossip}
W(k)=I-\g(k)\sum_{j\in N_{i}(k)}e_j(e_j-e_i)^T\qquad
\mbox{with probability $\frac{1}{m}$.}
\end{align}

Let $G_b^\infty$ be the infinite flow graph of the broadcast gossip model, 
and suppose that this graph has $\tau$ connected components, namely 
$S_1,\ldots,S_{\tau}$.
Using Theorem~\ref{thrm:extendedinfflow}, we have the following result.

\begin{lemma}\label{lemma:broadcastgossip}
The time-varying broadcast gossip model of~\eqref{eqn:broadcastgossip} is 
stable almost surely. Furthermore, any two 
agents are in the same ergodicity class if and only if they belong to 
the same connected component of $G_b^{\infty}$. In 
particular, the model is ergodic if and only if $G_r^{\infty}$ is connected.
\end{lemma} 
\begin{proof}
In view of Theorem~\ref{thrm:extendedinfflow}, it suffices to show that 
the broadcast gossip model has a common steady 
state $\pi>0$ in expectation and weak feedback property.
Since each agent is chosen uniformly at any time instance and 
the graph $G(k)$ is undirected, the (random) entries $W_{ij}(k)$ and 
$W_{ji}(k)$ have the same distribution. Therefore, the expected matrix 
$\EXP{W(k)}$ is a doubly stochastic matrix for any $k\geq 0$. 
Since $\g(k)\leq \gamma<1$, it follows that
$W_{ii}(k)\geq 1-\al(k)\geq\gamma$ for all $i\in[m]$ and all $k\ge0$. 
When a model satisfies $W_{ii}(k)\ge\g>0$ for all $i$ and $k$, 
then the model has weak feedback property with $\frac{\g}{m}$, as implied by 
Lemma 7 in~\cite{ErgodicityPaper}. 
\end{proof} 

The above result shows that no matter how the underlying network evolves 
with the time, when the broadcast gossip algorithm is applied to a time-varying
network the stability of the algorithm is guaranteed.
In fact, we can provide a characterization of the connected 
components $S_r$ for the infinite flow graph $G_b^\infty$. By 
Theorem~\ref{thrm:extendedinfflow}, it suffices to determine the infinite flow 
graph $\bar G_b^\infty$ of the expected model. 
A link $\{i,j\}$ is in the edge-set of the graph 
$\bar G_b^\infty$ if and only if
$\sum_{k=0}^\infty \left(\EXP{W_{ij}(k)} +\EXP{W_{ji}(k)}\right)=\infty$.
By~\eqref{eqn:broadcastgossip}, we have $\EXP{W_{ij}(k)}=\frac{1}{m}\g(k)$
if $j\in N(k)$ and otherwise $\EXP{W_{ij}(k)}=0$.
Thus, $\{i,j\}\in\bar G_b^\infty$ if and only if
$\sum_{k: \{i,j\}\in \E(k)}\,\g(k)=\infty.$

Two instances of the time-varying broadcast gossip algorithm that might be of 
practical interest are: (1) The case when $G(k)=G$ for all $k\geq 0$. 
Then, the random model is ergodic if and only if $G$ is connected and 
$\sum_{k=0}^{\infty}\g(k)=\infty$. (2) The case when 
the sequence $\{\g(k)\}$ is also bounded below i.e., $\g(k)\in[\g_b,\gamma]$ 
with $0<\g_b\leq \gamma<1$. Then, the model is ergodic if and only 
if, in the sequence $\{G(k)\}$, there are infinitely many edges between $S$ 
and $\bar{S}$ for any nonempty $S\subset[m]$.

\subsection{Link Failure Models}\label{sec:failure}
The application in this section is motivated by the work in~\cite{KarMoura07} 
where the ergodicity of a random link failure model has been 
considered.
However, the link failure model in~\cite{KarMoura07} corresponds to just a 
random model in our setting. Here, we assume that we have an 
underlying random model and that there is another random process that models
link failure in the random model. We use $\W$ to denote the underlying random 
model, as in Eq.~\eqref{eqn:dynsys}.
We let $\{F(k)\}$ denote a link failure process, which is independent of the 
underlying model $\W$. Basically, the failure process 
\textit{reduces the information flow between agents} in the underlying 
random model $\W$. For the failure process, we have either
$F_{ij}(k)=0$ or $F_{ij}(k)=1$ for all $i,j\in[m]$ and $k\geq 0$, so that 
$\Fp$ is a binary matrix sequence.
We define the \textit{link-failure model} as 
the random model $\{U(k)\}$ given by
\begin{align}\label{eqn:linkfailed}
U(k)=W(k)\cdot (ee^T-F(k))+\diag([W(k)\cdot F(k)]e),
\end{align}
where ``$\cdot$'' denotes the element-wise product of two matrices.    
To illustrate this model, suppose that we have  a random model 
$\W$ and suppose that each entry $W_{ij}(k)$ is set to zero (fails), 
when $F_{ij}(k)=1$. In this way, $F(k)$ induces 
a failure pattern on $W(k)$. The term $W(k)\cdot (ee^T-F(k))$ in 
Eq.~\eqref{eqn:linkfailed} reflects this effect. Thus, 
$W(k)\cdot (ee^T-F(k))$ does not have some of the entries of $W(k)$.
This lack is compensated by the feedback term which is equal to the sum of 
the failed links, the term 
$\diag([W(k)\cdot F(k)]e)$. This is the same as adding 
$\sum_{j\not= i}[W(k)\cdot F(k)]_{ij}$ to the self-feedback weight $W_{ii}(k)$
of agent $i$ at time $k$ in order to ensure
the stochasticity of $U(k)$. 

Now, let us define \textit{feedback property}.
A random model $\W$ has feedback property if there is $\gamma>0$ such that 
$\EXP{W_{ii}(k)W_{ij}(k)+W_{jj}(k)W_{ji}(k)}
\geq \gamma \EXP{W_{ij}(k)+W_{ji}(k)}$ for any $k\geq 0$ and $i,j\in[m]$ 
with $i\not=j$. In general, this property is stronger than weak feedback 
property, as proved in~\cite{ErgodicityPaper}. 

Our discussion will be focused on a special class of link failure processes,
which are introduced in the following definition. 
 
\begin{definition}\label{def:uniflinkfail} 
A \textit{uniform link-failure process} is a process $\Fp$ such that: 
\begin{enumerate}[(a)]
\item 
The random variables $\{F_{ij}(k)\mid i,j\in[m],\ i\not=j\}$ 
are binary i.i.d.\ for any fixed $k\ge0$. 
\item 
The process $\{F(k)\}$ is an independent process in time. 
\end{enumerate}
\end{definition}

Note that the i.i.d.\  condition in Definition~\ref{def:uniflinkfail} is 
assumed for a fixed time. Therefore, the uniform link-failure model can 
have a time-dependent distribution but for any given time the distribution of 
the link-failure should be identical across the different edges. 

For the uniform-link failure process, we have the following result.

\begin{lemma}\label{lemma:linkfail}
Let $\{W(k)\}$ be an independent model with a common steady state $\pi>0$ in 
expectation and feedback property. Let $\{F(k)\}$ 
be a uniform-link failure process that is independent of $\{W(k)\}$.
Then, the failure model $\{U(k)\}$ is ergodic if and only if 
$\sum_{k=0}^{\infty}(1-p_k)W_{S}(k)=\infty$ for any nonempty $S~\subset[m]$,
where $p_k=\Pr(F_{ij}(k)=1)$. 
\end{lemma}
\begin{proof}
By the definition of $\{U(k)\}$ in~\eqref{eqn:linkfailed},
the failure model $\{U(k)\}$ is also independent 
since both $\{W(k)\}$ and $\{F(k)\}$ are independent.
Then, for $i\not=j$ and for any $k\geq 0$, we have 
\begin{align}\label{eqn:expu}
\EXP{U_{ij}(k)}=\EXP{W_{ij}(k)(1-F_{ij}(k))}=(1-p_k)\EXP{W_{ij}(k)},
\end{align}
where the last equality holds since $W_{ij}(k)$ and $F_{ij}(k)$ 
are independent, and $\EXP{F_{ij}(k)}=p_k$. 
By summing the relations in~\eqref{eqn:expu} over $j\ne i$ for a fixed $i$,
we obtain $\sum_{j\not=i}\EXP{U_{ij}(k)}=(1-p_k)\sum_{j\not=i}\EXP{W_{ij}(k)}$,
which by stochasticity of $W(k)$ implies 
$\sum_{j\not=i}\EXP{U_{ij}(k)}=(1-p_k)(1-\EXP{W_{ii}(k)}).$
Since  $U(k)$ is stochastic, it follows that
\[\EXP{U_{ii(k)}}=1-\sum_{j\not=i}\EXP{U_{ij}(k)}
=p_k+(1-p_k)\EXP{W_{ii}(k)}.\]
From the preceding relation and Eq.~\eqref{eqn:expu},
in matrix notation, the following relation holds:
\begin{align}\label{eqn:expfailedmodel}
\EXP{U(k)}=p_k I+(1-p_k)\EXP{W(k)}\qquad\hbox{for all $k$}.
\end{align}
Since $\pi$ is a common steady state of $\{\EXP{W(k)}\}$, from
Eq.\eqref{eqn:expfailedmodel} we obtain
$\pi^T\EXP{U(k)}=\pi^T$, thus showing that
$\pi>0$ is also a common steady state for $\{U(k)\}$ in expectation.

We next show that $U(k)$ has feedback property. 
By the definition of $U(k)$, 
$U_{ii}(k)\geq W_{ii}(k)$ for all $i\in[m]$ and $k\geq 0$. 
Hence, $\EXP{U_{ii}(k)U_{ij}(k)}\geq \EXP{W_{ii}(k)U_{ij}(k)}$.  
Since $\{F(k)\}$ and $\{W(k)\}$ are independent, we have
\begin{align*}
\EXP{W_{ii}(k)U_{ij}(k)}
&=\EXP{\EXP{W_{ii}(k)U_{ij}(k)\mid F_{ij}(k)=0}}
=\EXP{\EXP{W_{ii}(k)W_{ij}(k)\mid F_{ij}(k)=0}}\cr 
&=(1-p_k)\EXP{W_{ii}(k)W_{ij}(k)}.
\end{align*}
A similar relation holds for $\EXP{U_{jj}(k)U_{ji}(k)}$. 
By the feedback property of $\{W(k)\}$, we have 
\begin{align*}
&\EXP{U_{ii}(k)U_{ij}(k)+U_{jj}(k)U_{ji}(k)}\geq (1-p_k)\gamma 
\EXP{W_{ij}(k)+W_{ji}(k)}=\gamma\EXP{U_{ij}(k)+U_{ji}(k)},\end{align*}
where the last equality follows from Eq.~\eqref{eqn:expfailedmodel},
and $\gamma>0$ is the feedback constant for $\{W(k)\}$.
Thus, $\{U(k)\}$ has feedback property with the same constant 
$\gamma$ as the model $\W$. Hence, the model $\{U(k)\}$  satisfies 
the assumptions of Theorem~\ref{thm:infflowthm}, so 
the model $\{U(k)\}$ is ergodic if and only if 
$\sum_{k=0}^{\infty}\EXP{U_S(k)}=\infty$ for any nontrivial $S\subset[m]$. 
By Eq.~\eqref{eqn:expfailedmodel} we have
$ \EXP{U_S(k)}=(1-p_k)\EXP{W_S(k)}$, implying that  
$\{U(k)\}$ is ergodic if and only if 
$\sum_{k=0}^{\infty}(1-p_k)\EXP{W_S(k)}=\infty$ for any nontrivial 
$S\subset[m]$. 
\end{proof}
Lemma~\ref{lemma:linkfail} shows that 
the severity of a uniform link failure process cannot cause instability 
in the system. When the failure probabilities $p_k$ are bounded away from~1 
uniformly, i.e., $p_k\leq \bar p$ for all $k$ and some $\bar p<1$, 
it can be seen that $\sum_{k=0}^{\infty}(1-p_k)\EXP{W_S(k)}=\infty$ 
if and only if $\sum_{k=0}^{\infty}\EXP{W_S(k)}{=\infty}$. In this case, 
by Lemma~\ref{lemma:linkfail} the following result is valid: the failure model 
$\{U(k)\}$ is ergodic if and only if the original model $\{W(k)\}$ is ergodic.

\section{Conclusion}\label{sec:conclusion}
In this paper, we have studied the limiting behavior 
of time-varying dynamics driven by random stochastic matrices. We have 
introduced the concept of $\ell_1$-approximation of a chain and have shown
that, for certain chains, such approximations preserve the limiting behavior 
of the original chains. We have also introduced the class $\M$ of stochastic 
chains to which the infinite flow theorem is applicable, which non-trivially
extends the class of models originally covered by this 
theorem~\cite{ErgodicityPaper}.
Finally, we have identified a certain class of independent random models that 
are stable almost surely. Moreover, we characterized the equilibrium points of 
these models by looking  at their infinite flow graphs. 
Finally, we have applied our main result to a broadcast gossip algorithm over 
a time-varying network and to a link-failure model.
\bibliographystyle{amsplain}
\bibliography{approx}
\end{document}